\documentclass[a4paper,11pt]{article}

\usepackage{amsfonts,qtree,stmaryrd,textcomp}
\usepackage{pifont,amssymb,amsmath,amsthm,tabularx,graphicx}
\usepackage{tree-dvips,pictexwd,dcpic}
\usepackage{bbm}
\usepackage{bm}

\usepackage{stmaryrd}

\usepackage[T1]{fontenc}
\usepackage[latin1]{inputenc}
\usepackage[text={17cm,26cm},centering]{geometry}

\usepackage{etex}
\usepackage{fancyhdr}
\usepackage{graphicx}
\usepackage{enumitem}
\usepackage{amsmath}
\usepackage[bbgreekl]{mathbbol}
\usepackage{amssymb}
\usepackage{amsthm}
\usepackage[all]{xy}

\usepackage{epigraph}
\RequirePackage{bussproofs}

\usepackage{mathabx}

\usepackage{scalerel,stackengine}

\usepackage{framed}
\usepackage{mathtools}
\usepackage{url}
\usepackage{proof}
\usepackage{xspace}
\usepackage{listings}
\usepackage{wrapfig}
\usepackage{tikz}
\usepackage{tikz-cd}

\usepackage{relsize}

\usetikzlibrary{positioning}
\usetikzlibrary{shapes}

\usetikzlibrary{arrows}
\usetikzlibrary{calc}
\usepackage{tikz-3dplot}
\usetikzlibrary{decorations.pathmorphing}

\usepackage{amsfonts,amssymb,amsmath,qtree,stmaryrd,textcomp, amsthm}

\usetikzlibrary{decorations.pathreplacing}
\tikzset{axis/.style={&lt;-&gt;}}

\stackMath
\newcommand\reallywidehat[1]{%
\savestack{\tmpbox}{\stretchto{%
  \scaleto{%
    \scalerel*[\widthof{\ensuremath{#1}}]{\kern-.6pt\bigwedge\kern-.6pt}%
    {\rule[-\textheight/2]{1ex}{\textheight}}
  }{\textheight}%
}{0.5ex}}%
\stackon[1pt]{#1}{\tmpbox}%
}
 \definecolor{MyBlue}{rgb}{0.05, 0.25, 0.65}
 
 \definecolor{MyRed}{rgb}{0.90, 0.05, 0.05}
 
\definecolor{MyGreen}{rgb}{0.05, 0.90, 0.05}

\newcommand{\B}{\boldsymbol}
\newcommand{\C}[1]{\mathcal{#1}}
\newcommand{\D}[1]{\mathbb{#1}}

\usepackage{ mathrsfs }

\newtheorem{theorem}{Theorem}[section]
\newtheorem{proposition}[theorem]{Proposition}
\newtheorem{lemma}[theorem]{Lemma}
\newtheorem{corollary}[theorem]{Corollary}
\newtheorem{remark}[theorem]{Remark}

\newtheorem{definition}[theorem]{Definition}

\newcommand{\Nat}{{\mathbb N}}
\newcommand{\Real}{{\mathbb R}}

\newcommand{\id}{\mathrm{id}}

\newcommand{\Const}{\mathrm{Const}}

\newcommand{\Bic}{\mathrm{Bic}}

\newcommand{\BISH}{\mathrm{BISH}}

\newcommand{\CST}{\mathrm{CST}}

\newcommand{\TOT}{\Leftrightarrow}

\newcommand{\To}{\Rightarrow}

\newcommand{\sto}{\rightsquigarrow}

\newcommand{\MLTT}{\mathrm{MLTT}} 
\newcommand{\CZF}{\mathrm{CZF}}

\newcommand{\CM}{\mathrm{CM}}

\newcommand{\prb}{\textnormal{\textbf{pr}}}

\newcommand{\pr}{\textnormal{\texttt{pr}}}

\newcommand{\BST}{\mathrm{BST}}

\newcommand{\Set}{\mathrm{\mathbf{Set}}}

\newcommand{\op}{\mathrm{op}}

\newcommand{\Hom}{\mathrm{Hom}}

\newcommand{\f}{\textnormal{\texttt{f}}}

\newcommand{\emptys}{\slash{\hspace{-2mm}}\Box}

\newcommand{\fV}{\D V_0^{\f}}

\newcommand{\Ineq}{\textnormal{\texttt{Ineq}}}

\newcommand{\SetIneq}{\textnormal{\textbf{SetIneq}}}
\newcommand{\SetComplSep}{\textnormal{\textbf{SetComplSep}}}
\newcommand{\SetStrong}{\textnormal{\textbf{SetStrong}}}

\newcommand{\BMT}{\mathrm{BMT}} 

\newcommand{\BCMT}{\mathrm{BCMT}}

\newcommand{\SetAffine}{\textnormal{\textbf{SetAffine}}}

\newcommand{\Aff}{\mathbb{Af}}

\newcommand{\Frg}{\textnormal{\texttt{Frg}}}

\newcommand{\Free}{\textnormal{\texttt{Free}}}

\newcommand*\circled[2][1.6]{\tikz[baseline=(char.base)]{
    \node[shape=circle, draw, inner sep=1pt, 
        minimum height={\f@size*#1},] (char) {\vphantom{WAH1g}#2};}}

\newcommand{\Top}{\textnormal{\textbf{Top}}}

\newcommand{\crTop}{\textnormal{\textbf{crTop}}}

\newcommand{\FunSpace}{\textnormal{\textbf{FunSpace}}}

\newcommand{\Emb}{\textnormal{\texttt{Emb}}}

\newcommand{\Dual}{\textnormal{\texttt{Dual}}}

\begin{document}

\date{}

\title{\textbf{Sets completely separated by functions in Bishop Set Theory}}

\author{Iosif Petrakis\\	
Mathematics Institute, Ludwig-Maximilians-Universit\"{a}t M\"{u}nchen\\
petrakis@math.lmu.de}  

%





\maketitle

\begin{abstract}
\noindent 
Within Bishop Set Theory, a reconstruction of Bishop's theory of sets,
we study the so-called completely separated sets, that is sets equipped with a positive notion of an inequality,
induced by a given set of real-valued functions. We introduce the notion of a global family of
completely separated sets over an index-completely separated set, and we 
describe its Sigma- and Pi-set. The free completely separated set on a given set is also presented. Purely set-theoretic
versions of the classical Stone-\v{C}ech theorem and the Tychonoff embedding theorem for completely regular spaces 
are given, replacing topological spaces with function spaces and completely regular spaces with completely 
separated sets.\\[2mm]
\end{abstract}


\section{Introduction}
\label{sec: intro}

Constructive mathematics $(\CM)$, although hard to define, is not only mathematics within intuitionisitic logic. It also 
encompasses certain ``attitudes'' towards the definition of mathematical concepts, in order to 
reveal the computational content of mathematical proofs of theorems concerning these concepts. Some of these attitudes,
which are not uniformly and consistently followed by the practitioners of $\CM$ though, are the following:
\begin{itemize}
 \item \textit{Predicativity}, that is the avoidance of circular, or impredicative definitions, 
 which often amounts to the avoidance of 
 quantification over proper classes in the definition of a mathematical concept.
 \item The addition of \textit{witnessing information} so that even weak choice
principles, such as countable choice\footnote{An approach to $\CM$ without countable choice was spearheaded by
Richman~\cite{Ri01}.}, are avoided in proofs. Bishop's moduli of convergence, of continuity, of
differentiability etc., in~\cite{Bi67} are examples of added witnesses to convergence, continuity, differentiability etc.
\item Preference for \textit{positive definitions} of concepts 
over negative ones, that is definitions using negation. The most characteristic example of the positive approach to 
mathematical concepts within $\CM$, which goes back to Brouwer, 
is the notion of inequality, or apartness relation on a set.
\item Preference for \textit{function-based concepts} over set-based ones, as 
functions suit better to $\CM$ than sets. Characteristic examples of this attitude are the definition of an 
integrable complemented set through the integrability of its partial characteristic function in 
constructive measure theory (see~\cite{BC72, BB85}) and the reduction of the topology of open sets to 
the topology of functions in constructive topology of Bishop spaces (see~\cite{Br12, Is13},~\cite{Pe15}-\cite{Pe22b} 
and~\cite{Pe23}).

\end{itemize}

In this paper we highlight the combination of the last two attitudes in the study of sets with an inequality,
determined in a positive way from a given set of functions. As many results shown here reveal an analogy between these
sets with the completely regular
topological spaces, we call them \textit{completely separated sets}.
Although the notion of an inequality induced by real-valued functions is implicitly used
by Bishop~\cite{Bi67}, p.~66, in his definition of a complemented subset, 
an elaborate study of sets equipped with such an inequality is missing. Ruitenburg's work~\cite{Ru91}, in this journal,
on inequalities in $\CM$ is set-based.

The various distinctions revealed within $\CM$ are also reflected in the 
various categories of sets that can be defined within $\CM$. 
Besides the standard category $\Set$ of sets $(X, =_X)$ and functions, we have
the category $\SetIneq$ of sets with an inequality (see Definition~\ref{def: ineq})
and strongly extensional functions as arrows (see Definition~\ref{def: fapartness}). 
This category is the ``universe'' of
Bishop-Cheng measure theory $(\BCMT)$, introduced in~\cite{BC72} and extended significantly in~\cite{BB85}.
In~\cite{Pe22c} we also introduce the category of strong sets
$\SetStrong$, a subcategory of $\SetIneq$, where the inequalities considered are equivalent to the \textit{strong negation}
of the corresponding equalities, a positive and strong counterpart to the standard \textit{weak
negation}. In the category $\SetComplSep$ of sets completely separated by 
real-valued functions on them (see Definition~\ref{def: fsets}), a subcategory of $\SetIneq$, the 
given inequalities and equalities are equivalent to ones induced by a given set of real-valued functions.
The importance of these inequalities lies in the complete avoidance of negation in their definition and in the proof
of their basic properties. The category $\SetComplSep$ is the 
`universe'' of Bishop measure theory\footnote{For the relation between $\BMT$ and $\BCMT$ 
see~\cite{Pe20, PW22, PZ22}.}$(\BMT)$, developed in~\cite{Bi67}.

We work within \textit{Bishop Set Theory} $(\BST)$, an informal, constructive theory of totalities and 
assignment routines, elaborated in~\cite{Pe19d, Pe20, Pe21, Pe22a}, that serves as a ``completion'' of Bishop's original
theory of sets in~\cite{Bi67, BB85}. $\BST$ highlights fundamental notions 
that were suppressed by Bishop in his account of the set theory 
underlying Bishop-style constructive mathematics $(\BISH)$, and serves as an intermediate step between
Bishop's informal theory of sets 
and an \textit{adequate} and \textit{faithful}
formalisation of $\BISH$, in Feferman's sense~\cite{Fe79}. To assure faithfulness,
we use concepts or principles that appear, explicitly or implicitly, in $\BISH$.
The features of $\BST$ in~\cite{Pe20} that somehow ``complete'' Bishop's original theory of sets are: the
explicit use of an open-ended universe $\D V_0$ of predicative sets, a clear distinction between sets and proper classes,
such as $\D V_0$, the explicit use of dependent operations, and the elaboration of the theory of set-indexed 
families of sets. Similarly to Martin-L\"of Type Theory $(\MLTT)$ (see for example~\cite{ML75, ML98}), $\BST$ 
behaves like a high-level programming language\footnote{The type-theoretic interpretation of Bishop's set theory into the theory
of setoids (see especially the work of Palmgren~\cite{Pa05}-\cite{PW14}) has become nowadays the standard
way to understand Bishop sets. For an analysis of the relation between intensional $\MLTT$ and Bishop's
theory of sets see~\cite{Pe20}, Chapter 1. Other formal systems for $\BISH$ are Myhill's Constructive Set Theory $(\CST)$,
introduced in~\cite{My75}, and Aczel's system $\CZF$ (see~\cite{AR10}). Palmgren's work~\cite{Pa12} and Coquand's
work~\cite{Co17} are categorical approaches to Bishop
sets.}.

We structure this paper as follows:
\begin{itemize}
 \item In section~\ref{sec: ineqs} we introduce the category $\SetIneq$ of sets with an inequality and strongly
 extensional functions, and we define the Sigma-and the Pi-set of a family of sets
 with an inequality over an index-set with an inequality.

 \item In section~\ref{sec: fineqs} we define the canonical equality $=_{(X,F)}$ and inequality $\neq_{(X,F)}$ 
 induced on a set $(X, =_X)$ by  an extensional subset $F$ of the real-valued functions $\D F(X)$ on $X$.
 The extensionality and the tightness of the inequality $\neq_{(X,F)}$ avoid completely the use of negation.
 Basic properties of these inequalities are shown.
 
 \item In section~\ref{sec: scs} we introduce the category $\SetComplSep$ of completely separated sets, a full subcategory
 of $\SetIneq$, and the category $\SetAffine$ of affine sets, a subcategory of $\SetComplSep$ with affine arrows only.
 We define the notion of a family of completely separated sets over an index-completely separated set\footnote{All notions
 of families of sets defined here are not studied in~\cite{Pe20}.}, we describe 
 its corresponding  Pi-set, and we provide a sufficient condition, in order to get a restricted form of a Sigma-set for it
 (Proposition~\ref{prp: fcl3}). By introducing the notion of a global family of completely separated sets over an 
 index-completely separated set, we manage to describe its Sigma-set as a completely separated set, and to generalise the
 notion of a strongly extensional function to dependent functions. The second projection of the Sigma-set of such a
 global family is shown to be a strongly extensional dependent function (Corollary~\ref{cor: sepr2}(ii)).

\item In section~\ref{sec: free} we define the free completely separated set 
$\varepsilon \B X$ on a given set $(X, =_X)$,
 we prove its universal property, and we show that the functor $\Free \colon \Set \to \SetComplSep$ is left adjoint to  
 the corresponding forgetful functor $\Frg \colon \SetComplSep \to \Set$. Theorem~\ref{thm: free1} corresponds 
 to the type-theoretic  fact that the setoid $(X, =_X)$, where $=_X$ is the equality type family on the type $X$
 in intensional $\MLTT$, is the free setoid on the type $X$.
 
 \item In section~\ref{sec: SC} we prove a  purely set-theoretic version of the Stone-\v{C}ech theorem in 
 classical topology, according to which, to any topological space corresponds 
 a completely regular one such that the two spaces have isomorphic rings of continuous,
 real-valued functions, and the corresponding functor is a reflector. Replacing topological spaces with function spaces, that is
 triplets $(X, =_X ; F)$ as above, and completely regular spaces with completely separated sets, we correspond 
 to any function space a completely separated set with the same carrier set and (separating) set of functions
 (Theorem~\ref{thm: sc}). The corresponding functor $\rho \colon \FunSpace \to \SetAffine$ is shown to be a reflector, 
 hence left adjoint to the corresponding embedding functor $\Emb \colon \SetAffine \to \FunSpace$. Moreover,
 $\Emb$ is also left adjoint to $\rho$ (Proposition~\ref{prp: rho}).

 \item In section~\ref{sec: TET} we prove a  purely set-theoretic version of the Tychonoff embedding theorem in 
 classical topology, according to which, a $T_1$ topological space is completely regular if and only if it 
is topologically embedded into a product of $[0, 1]$. According to Theorem~\ref{thm: tet}, if $(X, =_X ; F)$ is a 
function space, and if the induced inequality $\neq_{(X,F)}$ on $X$ is tight, then there is an affine embedding of the  
completely separated set $(X, =_X, \neq_{(X, F)} ; F)$ into the completely separated set $\B R^{F}$. Conversely,  
if $e \colon (X, =_X ; F) \to (\Real^{F}, =_{\Real^{F}} ; \bigotimes_{f \in F}\{\id_{\Real}\})$ is an affine 
embedding in $\FunSpace$, the  induced inequality $\neq_{(X,F)}$ on $X$ is tight. The latter result provides
a criterion for the generation of a completely separated set from a given function space.

\end{itemize}

For all notions and results from $\BST$ that are used here without explanation or proof, we refer 
to~\cite{Pe20}.  
For all categorical notions and facts that are used here without explanation or proof, we refer 
to~\cite{Aw10}. A formulation of category 
theory within $\BST$ is presented in~\cite{Pe22f}.

\section{Sets with an inequality}
\label{sec: ineqs}

In this section we introduce the category $\SetIneq$ of sets with an inequality.
We use a terminology on inequalities that is closer to~\cite{MRR88}, rather than to~\cite{BB85}. 
In~\cite{BB85}, p.~72, an inequality is what we call here an apartness relation. 
In~\cite{Sh21}, p.~31, an inequality is considered to be always symmetric. As the canonical inequality defined 
though strong negation in~\cite{Pe22c} is not in general symmetric, an abstract inequality here is a 
relation on a set that contradicts its given equality. In $\BST$ falsum is
defined as the formula $\bot := 0 =_{\Nat} 1$.

\begin{definition}\label{def: ineq}
If $(X, =_X)$ is a set, and let the following formulas with respect to a relation $x \neq_X y:$\\[1mm]
$(\Ineq_1) \  \forall_{x, y \in X}\big(x =_X y \ \& \ x \neq_X y \To \bot \big)$,\\[1mm]
$(\Ineq_2) \ \forall_{x, x{'}, y, y{'} \in X}\big(x =_X x{'} \ \& \ 
y =_X y{'} \ \& \ x \neq_X y \To x{'} \neq_X y{'}\big)$,\\[1mm]
$(\Ineq_3) \ \forall_{x, y \in X}\big(\neg(x \neq_X y) \To x =_X y\big)$,\\[1mm]
$(\Ineq_4)  \forall_{x, y \in X}\big(x \neq_X y \To y \neq_X x\big)$,\\[1mm]
$(\Ineq_5) \ \forall_{x, y \in X}\big(x \neq_X y \To \forall_{z \in X}(z \neq_X x \ \vee \ z \neq_X y)\big)$,\\[1mm]
$(\Ineq_6) \ \forall_{x, y \in X}\big(x =_X y \vee x \neq_X y\big)$.\\[1mm]
If $\Ineq_1$ is satisfied, we call $\neq_X$ an 
inequality on $X$, and the 
structure $\B X := (X, =_X, \neq_X)$ a set with an inequality. We also write $|\B X| := X$.
If $(\Ineq_6)$ is satisfied, then $\B X$ is 
called \textit{discrete}.
An inequality is called extensional, if $(\Ineq_2)$ holds, and it is called tight, if $(\Ineq_3)$ is satisfied.
An inequality satisfying $(\Ineq_4)$ and $(\Ineq_5)$ is called an apartness relation on $X$.
If $\B Y := (Y, =_Y, \neq_Y)$ is a set with an inequality,
a function\footnote{If $(X, =_X)$ and $(Y, =_Y)$ are sets, an assignment routine $f \colon X \sto Y$ is a function, if it 
preserves the corresponding equalities. The notion of a (non-dependent) assignment routine, that is of a routine that corresponds an 
element $f(x)$ of $Y$ to each element $x$ of $X$, is 
primitive in $\BST$.} $f \colon X \to Y$ is strongly extensional,
if, 
$$\forall_{x, y \in X}\big(f(x) \neq_Y f(y) \To x \neq_X y\big).$$
Let $\D F^{\neq}(\B X, \B Y)$ be the set\footnote{One needs the extensionality of $\neq_Y$ to show that 
the strong extensionality of $f \in \D F(X, Y)$ is an extensional property on $\D F(X, Y)$, and define then 
$\D F^{\neq}(\B X, \B Y)$ using the separation scheme. As $(\Ineq_2)$ is not considered here part of the definition of an 
inequality, we introduce $\D F^{\neq}(\B X, \B Y)$ independently from $\D F(X, Y)$.}   
of strongly extensional functions from $\B X$ to $\B Y$,
equipped with the pointwise equality. Furthermore, let
$\D V_0^{\neq}$ be the proper class of sets with an inequality, equipped with the equality
$$\B X =_{\D V_0^{\neq}} \B Y :\TOT \exists_{f \in \D F^{\neq}(\B X, \B Y)}\exists_{g \in \D F^{\neq}(\B Y, 
\B X)}\big((f, g) \colon X =_{\D V_0} Y\big).$$
Let $\SetIneq$ be the category of sets with an inequality and strongly extensional functions.
\end{definition}

An equality is always extensional on $X \times X$. There are non-extensional properties on a set, that is 
formulas $P(x)$, where $x$ is a variable of the set $X$, such that $P(x)$ and $x =_X y$ do not imply $P(y)$.
For example, let $n \in \Nat$, $q \in \D Q$ and $P_q(x)$, where $x$ is a variable of set $\Real$, defined by
 $P_q(x) := x_n =_{\D Q} q.$
 If $y \in \Real$ such that $y =_{\Real} x$, then it is not necessary that $y_n =_{\D Q} q$, if $ x_n =_{\D Q} q$.
It is not easy to give examples of non strongly extensional functions, although we cannot accept
in $\BISH$ that all functions are strongly extensional. For example, the strong extensionality of 
all functions from a metric space to itself is equivalent to Markov's principle (see~\cite{Di20}, p.~40).
Even to show that a constant function between sets with an inequality is strongly extensional, one needs 
intuitionistic,  and not minimal, logic.

\begin{remark}\label{rem: apartness1}
An apartness relation on a set $(X, =_X)$ is an extensional inequality.
\end{remark}

\begin{proof}
Let $x, y \in X$ with $x \neq_X y$, and $x{'}, y{'} \in X$ with $x{'} =_X x$ 
and $y{'} =_X y$. By $(\Ineq_5)$ we get $x{'} \neq_X x$, which is excluded from $(\Ineq_1)$, or $x{'} \neq_X y$,
which has to be the case. Hence, $y{'} \neq_X x{'}$, or $y{'} \neq_X y$. Since the latter is excluded similarly,
we get $y{'} \neq_X x{'}$, hence by $(\Ineq_4)$ we get $x{'} \neq_X y{'}$.
\end{proof}

For the canonical equalities of all sets mentioned here we refer to~\cite{Pe20}, Chapter 2, and to~\cite{BB85}.

\begin{definition}\label{def: canonical}
Let  $\B X := (X, =_X, \neq_X)$ and $\B Y := (Y, =_Y, \neq_Y)$ be in $\SetIneq$, and let $\big(A, i_A\big) \subseteq X$.
The canonical inequalities on the product $X \times Y$, the function space $\D F(X, Y)$ and $A$
are given, respectively, by
$$(x, y) \neq_{X \times Y} (x{'}, y{'}) :\TOT x \neq_X x{'} \vee y \neq_Y y{'},$$
$$f \neq_{\D F(X, Y)} g :\TOT \exists_{x \in X} \big[f(x) \neq_Y g(x)\big],$$
$$a \neq_A a{'} :\TOT i_A(a) \neq_X i_A(a{'}).$$
The canonical inequality on the set of reals $\Real$ is given by 
$a \neq_{\Real} b :\TOT |a - b| > 0 \TOT a > b \vee a < b$, which is  
a special case of the canonical inequality on a metric space $(Z, d)$, given by 
$z \neq_{(Z,d)} z{'} :\TOT d(z, z{'}) > 0$. 
Let
$\B R := (\Real, =_{\Real}, \neq_{\Real})$, $\B Z := (Z, =_{(Z,d)}, \neq_{(Z,d)})$, and  
$\B N := (\Nat, =_{\Nat}, \neq_{\Nat})$. If $\D 2 := \{x \in \Nat \mid 
x =_{\Nat} 0 \vee x =_{\Nat} 1\}$, let $\B 2 := (\D 2, =_{\D 2}, \neq_{\D 2})$, where the $($in$)$equality on $\D 2$ is 
induced by the $($in$)$equality on $\Nat$.
\end{definition}

The inequalities $a \neq_{\Real} b$, $z \neq_{(Z,d)} z{'}$, $\neq_{\Nat}$ and $\neq_{\D 2}$ are tight apartness relations.
Clearly, the projections $\pr_X, \pr_Y$ associated to $X \times Y$, and also, by definition, 
the embedding $i_A^X$ of a subset $A$ of $X$, 
are strongly extensional functions. The dependent assignment routines described next are primitive objects in $\BST$.

\begin{definition}\label{def: do}
If $(I, =_I)$ is a set and $\lambda_0 \colon I \sto \D V_0$ is an assignment routine that corresponds to each element 
$i \in I$ a set $\lambda_0(i)$, a dependent assignment routine $\Theta$
 assigns to every element $i \in I$ an element $\Theta_i$ of $\lambda_0(i)$. We denote by
 $\C D \C O(I, \lambda_0)$ the totality of dependent assignment routines over $I$ and 
$\lambda_0$, which becomes a set with an inequality, if we consider the following pointwise equality and inequality
$$\Theta =_{\C D \C O(I, 
\lambda_0)} \Theta{'} :\TOT 
\forall_{i \in I}\big(\Theta_i =_{\lambda_0(i)} \Theta{'}_i\big),$$
$$\Theta \neq_{\C D \C O(I, \lambda_0)} \Theta{'} :\TOT 
\exists_{i \in I}\big(\Theta_i \neq_{\lambda_0(i)} \Theta{'}_i\big),$$ 
where for the latter we suppose that every $\lambda_0(i)$ is equipped with an inequality $\neq_{\lambda_0(i)}$.
\end{definition}

A family of sets  indexed by some set $(I, =_I)$ is an
assignment routine $\lambda_0 : I \sto \D V_0$ that 
behaves like a function, that is if $i =_I j$, then $\lambda_0(i) =_{\D V_0} \lambda_0 (j)$.
A more explicit definition, which is due to Richman, is included in~\cite{BB85},
p.~78 (Problem 2), which is made precise in~\cite{Pe20} by highlighting 
the role of dependent assignment routines in its formulation. In accordance to the second
attitude described in the Introduction, 
this is a proof-relevant definition revealing the witnesses of the 
equality $\lambda_0(i) =_{\D V_0} \lambda_0 (j)$. 
We define the notion of a family of sets within  $\SetIneq$ similarly.

\begin{definition}\label{def: famofsets}
If $\B I := (I, =_I, \neq_I)$ is in $\SetIneq$, let the \textit{diagonal} $D(I)$ of $I$,
defined by
$$D(I) := \{(i,j) \in I \times I \mid i =_I j\}.$$
A \textit{family of sets with an inequality}\index{family of sets} indexed by $\B I$
is a pair $\Lambda := (\lambda_0, \lambda_1)$, where\index{$\lambda_0$}
$\lambda_0 \colon I \sto \D V_0^{\neq}$, 
$$\overline{\lambda}_0(i) := \big(\lambda_0(i), =_{\lambda_0(i)}, \neq_{\lambda_0(i)}\big),$$
for every $i \in I$,
and\index{$\lambda_1$} $\lambda_1$, a
\textit{modulus of function-likeness for}\index{modulus of function-likeness} $\lambda_0$, is 
a dependent operation
\[ \lambda_1 \colon \bigcurlywedge_{(i, j) \in D(I)}\D F^{\neq}\big(\lambda_0(i), \lambda_0(j)\big), 
\ \ \ \lambda_1(i, j) =: \lambda_{ij}, \ \ \ (i, j) \in D(I), \]
such that the \textit{transport maps}\index{transport map of a family of sets} $\lambda_{ij}$
\index{transport map of a family of sets}\index{$\lambda_{ij}$}
of $\boldmath \Lambda$ satisfy the following conditions:\\[1mm]
\normalfont (a) 
\itshape For every $i \in I$, we have that $\lambda_{ii} = \id_{\lambda_0(i)}$.\\[1mm]
\normalfont (b) 
\itshape If $i =_I j$ and $j =_I k$, the following triangle commutes
\begin{center}
\begin{tikzpicture}

\node (E) at (0,0) {$\lambda_0(j)$};
\node[right=of E] (F) {$\lambda_0(k).$};
\node [above=of E] (D) {$\lambda_0(i)$};

\draw[->] (E)--(F) node [midway,below] {$\lambda_{jk}$};
\draw[->] (D)--(E) node [midway,left] {$\lambda_{ij}$};
\draw[->] (D)--(F) node [midway,right] {$\ \lambda_{ik}$};

\end{tikzpicture}
\end{center}
If $\B X, \B Y$ are in $\SetIneq$, the \textit{constant} $\B I$-\textit{family of sets}\index{constant family of sets}
$\B X$\index{$C^X$} is the pair 
$\B C^X := (\B \lambda_0^X, \lambda_1^X)$, where $\lambda_0 (i) := X$, for every 
$i \in I$, and $\lambda_1 (i, j) := \id_X$, for every $(i, j) \in D(I)$. 
The $\B 2$-family $\boldmath \Lambda^{\D 2}(\B X, \B Y)$ of $\B X$ and $\B Y$ is defined by $\lambda_0 (0) := \B X$, 
$\lambda_0 (1) := \B Y$, $\lambda_{00} := \id_X$ and $\lambda_{11} := \id_Y$.
\end{definition}

If $i =_I j$, then $(\lambda_{ij}, \lambda_{ji}) \colon \B \lambda_0(i) =_{\D V_0^{\neq}} \B \lambda_0(j)$.
Next we describe the Sigma-set (or the exterior union, or the disjoint union) and the Pi-set (or the set of dependent 
functions) of a given family of sets with an inequality.

\begin{definition}\label{def: sigmaset}
Let $\boldmath \Lambda := (\lambda_0, \lambda_1)$ be an $\B I$-family of sets  with an inequality. Its 
Sigma-set
$$\sum_{I}\Lambda := \bigg(\sum_{i \in I}\lambda_0 (i), =_{\mathsmaller{\sum_{i \in I}\lambda_0 (i)}}, 
\neq_{\mathsmaller{\sum_{i \in I}\lambda_0 (i)}}\bigg) \in \SetIneq$$
is defined by
\[ w \in \sum_{i \in I}\lambda_0 (i) : 
\TOT \exists_{i \in I}\exists_{x \in \lambda_0 (i)}\big(w := (i, x)\big), \]
\[ (i, x) =_{\mathsmaller{\sum_{i \in I}\lambda_0 (i)}} (j, y) : \TOT i =_I j \ \& \ \lambda_{ij} (x) 
=_{\lambda_0 (j)} y,\]
\[(i,x) \neq_{\mathsmaller{\sum_{i \in I}\lambda_0 (i)}} (j, y) :\TOT i \neq_I j \ \vee \ \big(i =_I j \ \& \ 
\lambda_{ij}(x) \neq_{\lambda_0(j)} y \big). \]
The Sigma-set of the $\B 2$-family $\Lambda^{\D 2}(\B X, \B Y)$ of $\B X$ and $\B Y$
is their \textit{coproduct}\index{coproduct}.
The \textit{first projection}\index{first projection} on
$\sum_{i \in I}\lambda_0 (i)$ is the $($non-dependent$)$ operation\footnote{The global projection operations $\prb_1$ and $\prb_2$
are primitive operations in $\BST$.}
$$\pr_1^{\boldmath \Lambda} \colon \sum_{i \in I}\lambda_0 (i) \sto I, \ \ \ \  
\pr_1^{\boldmath \Lambda} (i, x) : = \prb_1 (i, x) := i, \ \ \ \ (i, x) \in \sum_{i \in I}\lambda_0 (i).$$
The
\textit{second projection}\index{second projection} on
$\sum_{i \in I}\lambda_0 (i)$ is the dependent operation\index{$\pr_2^{\Lambda}$}
$$\pr_2^{\boldmath \Lambda} \colon \bigcurlywedge_{(i,x) \in \sum_{i \in I}\lambda_0 (i)}\lambda_0(i),
\ \ \ \ \pr_2^{\boldmath \Lambda}(i,x) := \prb_2(i,x) := x, \ \ \ \ (i, x) \in \sum_{i \in I}\lambda_0 (i).$$
We write $\pr_1, \pr_2$, if $\boldmath \Lambda$ is clearly understood from the context.
The Pi-set of $\boldmath \Lambda$ 
$$\prod_{I}\Lambda := \bigg(\prod_{i \in I}\lambda_0 (i), =_{\mathsmaller{\prod_{i \in I}\lambda_0 (i)}}, 
\neq_{\mathsmaller{\prod_{i \in I}\lambda_0 (i)}}\bigg) \in \SetIneq$$
is defined by
$$\Theta \in \prod_{i \in I}\lambda_0(i) :\TOT \Theta \in \C D \C O(I, \lambda_0) \ \& \ 
\forall_{(i,j) \in D(I)}\big(\Theta_j 
=_{\lambda_0(j)} \lambda_{ij}(\Theta_i)\big),$$
and $=_{\mathsmaller{\prod_{i \in I}\lambda_0 (i)}}$ and $\neq_{\mathsmaller{\prod_{i \in I}\lambda_0 (i)}}$ are
the the canonical pointwise equality and inequality, respectively, inherited from 
$\C D \C O(I, \lambda_0)$.
If $\B X$ is in $\SetIneq$ and $C^{\B X}$ is the constant $\B I$-family $\B X$,
let
\[ \B X^{\B I} := \bigg(X ^I := \prod_{i \in I}X, =_{X^I}, \neq_{X^I}\bigg). \]

\end{definition}

Clearly, if $i =_I j$ and $\Theta \in \prod_{i \in I}\lambda_0(i)$, then $(i, \Theta_i) 
=_{\mathsmaller{\sum_{i \in I}\lambda_0 (i)}} (j, \Theta_j)$. 
Although the canonical inequality $=_{\mathsmaller{\sum_{i \in I}\lambda_0 (i)}}$ is an inequality on 
$\big(\sum_{i \in I}\lambda_0 (i), \neq_{\mathsmaller{\sum_{i \in I}\lambda_0 (i)}}\big)$, with respect to which 
$\pr_1^{\boldmath \Lambda}$ is a strongly extensional function, we need extra assumptions 
to make it an apartness relation. The proof of the next proposition makes heavy use of negation, something that 
we shall avoid in the subsequent sections, where inequalities induced by functions will be considered.


\begin{proposition}\label{prp: sigmaset1}
Let $\boldmath \Lambda := (\lambda_0, \lambda_1)$ be an $\B I$-family of sets with an inequality.
If $\B I$ is discrete, and $\neq_{\lambda_0(i)}$ is an apartness relation 
on $\lambda_0(i)$, for every 
$i \in I$,
then $ \neq_{\mathsmaller{\sum_{i \in I}\lambda_0 (i)}}$ 
is an apartness relation.
If $\B \lambda_0(i)$ 
is discrete, for every $i \in I$, then $\sum_{I}\Lambda$
is discrete, and if $\neq_I$ is tight and $\neq_{\lambda_0(i)}$ is tight, for every 
$i \in I$, then $\neq_{\mathsmaller{\sum_{i \in I}\lambda_0 (i)}} $ is tight.
 
\end{proposition}

\begin{proof}
The condition $(\Ineq_1)$ of Definition~\ref{def: ineq} is trivially satisfied. To show condition
$(\Ineq_4)$, we suppose
first that $i \neq_I j$, hence by the corresponding condition of $\neq_I$ we get $(j, y) 
\neq_{\mathsmaller{\sum_{i \in I}\lambda_0 (i)}} (i, x)$. If $i =_I j \ \& \ 
\lambda_{ij}(x) \neq_{\lambda_0(j)} y$, we show that $\lambda_{ji}(y) \neq_{\lambda_0(i)} x$. 
By the extensionality of $\neq_{\lambda_0(j)}$ 
(Remark~\ref{rem: apartness1}) 
the inequality $\lambda_{ij}(x) \neq_{\lambda_0(j)} y$ implies the inequality
$\lambda_{ij}(x) \neq_{\lambda_0(j)} \lambda_{ij}\big(\lambda_{ji}(y)\big)$, and since 
$\lambda_{ij}$ is strongly extensional, we get $x \neq_{\lambda_0(i)} \lambda_{ji}(y)$. To show 
condition $(\Ineq_5)$, let $(i,x) \neq_{\mathsmaller{\sum_{i \in I}\lambda_0 (i)}} (j, y)$, and let 
$(k, z) \in \sum_{i \in I}\lambda_0 (i)$. If $i \neq_I j$, then by condition $(\Ineq_5)$ of $\neq_I$ we get
$k \neq_I i$, or $k \neq_I j$, hence $(k,z) \neq_{\mathsmaller{\sum_{i \in I}\lambda_0 (i)}} (i, x)$, or
$(k,z) \neq_{\mathsmaller{\sum_{i \in I}\lambda_0 (i)}} (j, y)$. Suppose next $i =_I j \ \& \ 
\lambda_{ij}(x) \neq_{\lambda_0(j)} y$. As $\B I$ is discrete, $k \neq_I i$, or $k =_I i =_I j$.
If $k \neq_I i$, then what we want to show follows immediately. If $k =_I i =_I j$, then by the extensionality of 
$\neq_{\lambda_0(j)}$ and the strong extensionality of the transport map
$\lambda_{kj}$ we have that
\[ \lambda_{ij}(x) \neq_{\lambda_0(j)} y \To \lambda_{kj}\big(\lambda_{ik}(x)\big) \neq_{\lambda_0(j)} 
 \lambda_{kj}\big(\lambda_{jk}(y)\big) \To \lambda_{ik}(x) \neq_{\lambda_0(k)} \lambda_{jk}(y). \]
Hence, by condition $(\Ineq_5)$ of $\neq_{\lambda_0(k)}$ we get $\lambda_{ik}(x) \neq_{\lambda_0(k)} z$, or
$\lambda_{jk}(y) \neq_{\lambda_0(k)} z$, hence $(i,x) \neq_{\mathsmaller{\sum_{i \in I}\lambda_0 (i)}} (k, z)$, or
$(j,y) \neq_{\mathsmaller{\sum_{i \in I}\lambda_0 (i)}} (k, z)$. Let 
$\overline{\lambda}_0(i)$ be discrete, for every $i \in I$. We show that 
$(i,x) =_{\mathsmaller{\sum_{i \in I}\lambda_0 (i)}} (j, y)$, namely $i =_I j$ and $\lambda_{ij}(x) =_{\lambda_0(i)} y$, or 
$(i,x) \neq_{\mathsmaller{\sum_{i \in I}\lambda_0 (i)}} (j, y)$, that is 
$i \neq_I j$ or $i =_I j \ \& \ \lambda_{ij}(x) \neq_{\lambda_0(j)} y$. As $\B I$ is discrete, 
$i =_I j$, or $i \neq_I j$. In the first case, and since $\overline{\lambda}_0(j)$ is discrete, we get
$\lambda_{ij}(x) =_{\lambda_0(j)} y$ or 
$\lambda_{ij}(x) \neq_{\lambda_0(j)} y$, and what we want follows immediately. If $i \neq_I j$, we get
$(i,x) \neq_{\mathsmaller{\sum_{i \in I}\lambda_0 (i)}} (j, y)$.
Finally, we suppose that $\neq_I$ is tight and $\neq_{\lambda_0(i)}$ is tight, for every 
$i \in I$. Let $\neg\big[(i,x) \neq_{\mathsmaller{\sum_{i \in I}\lambda_0 (i)}} (j, y)\big]$, that is 
\[ \big[i \neq_I j \ \vee \ \big(i =_I j \ \& \ \lambda_{ij}(x) \neq_{\lambda_0(j)} y \big)\big] \To \bot. \]
From this hypothesis we get the conjunction\footnote{Here we use the logical implication
$\big((\phi \vee  \psi) \To \bot\big) \To [(\phi \To \bot) \ \& \ (\psi \To \bot)]$.}
\[ \big[i \neq_I j \To \bot\big] \ \ \& \ \ \big[\big(i =_I j \ \& \ \lambda_{ij}(x) \neq_{\lambda_0(j)} y \big) \To
\bot\big]. \]
By the tightness of $\neq_I$ we get $i =_I j$. The formula 
$\big(i =_I j \ \& \ \lambda_{ij}(x) \neq_{\lambda_0(j)} y \big) \To
\bot$ 
trivially implies 
$(i =_I j) \To  \big(\lambda_{ij}(x) \neq_{\lambda_0(j)} y \To \bot\big)$,
and as
its premiss $i =_I j$ is derived by the tightness of $\neq_I$, 
we get  
$\lambda_{ij}(x) \neq_{\lambda_0(j)} y \To \bot$. Since $\neq_{\lambda_0(i)}$ is tight, we conclude that 
$\lambda_{ij}(x) =_{\lambda_0(j)} y$, hence $(i, x) =_{\mathsmaller{\sum_{i \in I}\lambda_0 (i)}} (j, y)$.
\end{proof}

Notice that the above inequality on the Sigma-set of a family $\boldmath \Lambda$ does not give the canonical
inequality of the product when the constant family is considered. This will be resolved in section~\ref{sec: fineqs}
with the use of global families of completely separated sets. 
An abstract inequality induces a notion of disjoint subsets and
a complemented subset is a pair of disjoint subsets, that is a subset with a given ``complement''.
Complemented subsets are first-class citizens of Bishop-Cheng measure theory in~\cite{BC72, BB85}, a 
constructive version of Daniell's approach to measure theory, and of Bishop measure theory in~\cite{Bi67} 
(see~\cite{Pe19b, Pe22e, PW22, PZ22}).

\section{Equalities and inequalities induced by functions}
\label{sec: fineqs}


A completely positive notion of inequality is implicitly 
used in the definition of a complemented subset
in~\cite{Bi67}, p.~66. This is the inequality induced by a set of real-valued functions on a
given set. This concept, together with the corresponding notion of equality 
induced by such a set of functions, are the starting point of our study. 
\textit{Unless otherwise stated, in the rest of the paper $(X, =_X), (Y, =_Y), (I, =_I)$ are sets and $F, G, K$ 
are extensional subsets of $\D F(X), \D F(Y)$ and $\D F(I)$, respectively}\footnote{That is
there is an extensional property $P_F$ on $\D F(X)$ such that
$$F := \big\{f \in \D F(X) \mid P_F(f)\big\},$$
where the extensionality of $P_F$ is the property
$f =_{\D F(X)} g \ \& \ P_F(f) \To P_F(g)$, for every $f, g \in \D F(X)$. The
use of extensional subsets, instead of arbitrary subsets given
by their embeddings, is crucial to the definition of affine arrows in Definition~\ref{def: fsets}, 
as the expression $f \in F$,
which is in principle a judgment, in Martin-L\"of's sense, is replaced by the formula $P_F(f)$ (see also~\cite{Pe22d}).}.

\begin{definition}\label{def: fapartness}
Let $\B X$ be in $\SetIneq$.
The canonical equality $x =_{(X, F)} x{'}$ on $X$ induced by $F$ is defined by
$$x =_{(X,F)} x{'} :\TOT \forall_{f \in F}\big(f(x) =_{\Real} f(x{'})\big),$$
for every $x, x{'} \in X$, and the canonical inequality $x \neq_{(X, F)} x{'}$ on $X$ induced by $F$ is defined by
$$x \neq_{(X,F)} x{'} :\TOT \exists_{f \in F}\big(f(x) \neq_{\Real} f(x{'})\big),$$
for every $x, x{'} \in X$. We write $f \colon x \neq_{(X,F)} x{'}$ to denote that $f \in F$ witnesses the
inequality $x \neq_{(X,F)} x{'}$.
The induced inequality $x \neq_{(X, F)} x{'}$ is equal to the given inequality $\neq_X$ of $X$, if
$$x \neq_X x{'} \TOT x \neq_{(X, F)} x{'},$$
for every $x, x{'} \in X$.
%
The inequality $\neq_{(X,F)}$ is called  tight, if,  
 $$x =_{(X,F)} x{'} \To x =_X x{'},$$
 for every $x, x{'} \in X$, and in this case we call $F$ a separating set of functions on $X$, or we say that 
 $F$ separates the points of $X$. 
 If $\B Y$ in $\SetIneq$,
 and $\Phi$ is an extensional subset of $\D F(X, Y)$, the canonical $($in$)$equalities $(\neq_{(X,\Phi)}) =_{(X,\Phi)}$
 on $X$ induced by $\Phi$,
 and the  tightness of $\neq_{(X, \Phi)}$ are defined in a similar way. 
\end{definition}



%
\begin{remark}\label{rem: f1}
Let $(X, =_X, \neq_X)$ be a set with an inequality.\\[1mm]
\normalfont (i) 
\itshape For every $x, x{'} \in X$ we have that $x =_X x{'} \To x =_{(X, F)} x{'}$.\\[1mm]
\normalfont (ii) 
\itshape If $F$ separates the points of $X$, the equality $=_X$ is equal to the induced equality  $=_{(X, F)}$
on $X$.\\[1mm]
\normalfont (iii) 
\itshape The inequality $\neq_{(X, F)}$ is tight with respect to the equality $=_{(X, F)}$ on $X$.\\[1mm]
\normalfont (iv) 
\itshape The canonical inequality $x \neq_{(X, F)} x{'}$ on $X$ induced by $F$ is an apartness relation.\\[1mm]
\normalfont (v) 
\itshape If $\neq_X$ is equal to $\neq_{(X, F)}$, then $\neq_X$ is an apartness relation.\\[1mm]
\normalfont (vi) 
\itshape  $\neq_{(X, F)}$ is tight, according to Definition~\ref{def: fapartness}, if and only if it is 
tight according to Definition~\ref{def: ineq}.\\[1mm]
\normalfont (vii) 
\itshape If each $f \in F$ is strongly extensional, then $x \neq_{(X, F)} x{'} \To x \neq_X x{'}$, 
for every $x, x{'} \in X$.
\end{remark}

\begin{proof}
Cases (i)-(iii) are trivial, and (v) follows (iv). The proof of
(iv) relies on the fact that the canonical inequality $\neq_{\Real}$ is an apartness relation. 
For the proof of (vi), suppose that $x =_{(X, F)} x{'} \To x =_X x{'}$, and let $\neg(x \neq_{(X, F)} x{'})$. Then 
$x =_{(X, F)} x{'}$, using the tightness of $\neq_{\Real}$: if $f \in F$ with $f(x) \neq_{\Real} f(x{'})$,
then by our hypothesis we get a contradiction, hence $\neg(\big(f(x) \neq_{\Real} f(x{'})\big)$. By the tightness 
of $\neq_{\Real}$ we get $f(x) =_{\Real} f(x{'})$, and we get the required equality $x =_X x{'}$. For the converse 
implication, let $\neg(x \neq_{(X, F)} x{'}) \To x =_X x{'}$, and $x =_{(X, F)} x{'}$. To show $x =_X x{'}$,
it suffices to 
show $\neg(x \neq_{(X, F)} x{'})$. Clearly, the hypothesis $x \neq_{(X,F)} x{'}$ together with the assumption $x =_{(X, F)} x{'}$
give the 
required contradiction. The proof of case (vii) is immediate.
\end{proof}

By Remark~\ref{rem: f1}(i), the 
given equality 
$x =_X x{'}$ is ``smaller''
than every induced equality $x =_{(X, F)} x{'}$. 
The inequality $\neq_{(X,F)}$ is an inequality, but no negation of some sort is used in its definition,
and the proof of its extensionality avoids negation, as it relies on the extensionality of $\neq_{\Real}$.
By its extensionality,
$\neq_{(X,F)}$ provides a fully positive definition of the 
\textit{extensional empty subset} of $X$ induced by $F$, namely
$$ \emptys_{(X, F)} := \big\{x \in X \mid x \neq_{(X,F)} x \big\}.$$
With the help of the extensional empty subset a positive definition of the property ``a subset is empty'', defined
negatively as ``it is not inhabited'' in~\cite{BR87}, p.~8, is possible. Following~\cite{PW22},
if $(A, i_A) \subseteq X$, we call $A$ empty, if $A \subseteq \emptys_{(X, F)}$ in $\C P(X)$.
Similarly, the tightness of $\neq_{(X,F)}$ is formulated positively. Notice that the contrapositive of tightness,
$\neg(x =_X x{'}) \To \exists_{f\in F}\big(f(x) \neq_{\Real} f(x{'})\big)$,
is the standard, classical property of separation of points $x, x{'} \in X$ from $F$.
If $F$ is a Bishop topology of functions 
(see~\cite{Pe15,Pe20b,Pe21}), then $x =_{(X,F)} x{'}$ and $x \neq_{(X,F)} x{'}$ are the canonical equality and 
inequality  of a Bishop space, respectively.
The inequality $a \neq_{\Real} b$ is equal to $\neq_{(\Real, \Bic(\Real))}$, 
where $\Bic(\Real)$ is the topology of Bishop-continuous functions of type $\Real \to \Real$ 
(see~\cite{Pe15}, Proposition 5.1.2, where even pointwise continuous functions are shown to separate the
reals). But we can go even further, noticing that the following extensional subset of $\D F(\Real)$ 
$$\{\id_{\Real}\} := \big\{f \in \D F(\Real) \mid f =_{\D F(\Real)} \id_{\Real}\big\}$$
separates the points of $\Real$. If $(Z, d)$ is a metric space, let the set 
$$U_0(Z) := \{d_{z} \in \D F(Z) \mid z \in Z\},$$
where $d_z (z{'}) := d(z, z{'})$, for every $z{'} \in Z$. Then $z \neq_{(Z, d)} z{'} \TOT z \neq_{U_0(Z)} z{'}$,
for every $z, z{'} \in Z$. The family $\Const(X)$ of constant real-valued functions on $X$ is not separating,
in general. The following remark is straightforward to show.

 \begin{remark}\label{rem: f2}
Let $(X, =_X, \neq_X)$, $F, F{'}$ extensional subsets of $\D F(X)$ with $F \subseteq F{'}$, and $x, x{'} \in X$.\\[1mm]
\normalfont (i) 
\itshape $x =_{(X, F{'})} x{'} \To x =_{(X, F)} x{'}$.\\[1mm]
\normalfont (ii) 
\itshape $x \neq_{(X, F)} x{'} \To x \neq_{(X, F{'})} x{'}$.\\[1mm]
\normalfont (iii) 
\itshape If $\neq_{(X, F)}$ is tight, then $\neq_{(X, F{'})}$ is tight.\\[1mm]
\normalfont (iv) 
\itshape $ \emptys_{(X, F)} \subseteq  \emptys_{(X, F{'})}$.
\end{remark}

Hence, $=_{(X, \D F(X))}$ is the ``smallest'' equality on $X$ induced by real-valued 
functions on $(X, =_X)$, and $\neq_{(X, \D F(X))}$ is the largest inequality on $X$ induced by such functions.

 \section{Sets completely separated by functions}
\label{sec: scs}

 Next we introduce the category $\SetComplSep$ of sets completely separated by functions.
 Roughly speaking, a set with an inequality
 $(X, =_X, \neq_X)$ is completely separated, if there is an extensional subset $F$ of $\D F(X)$ such that 
 $=_X$ and $\neq_X$ ``are'' the ones induced by $F$.
 Consequently, $\neq_X$ is a tight apartness relation, and this fact is formulated in a completely 
 positive, negation-free framework.
 In order to avoid quantification over the powerset of $\D F(X)$, which is a
 proper class, we take $F$ to be part of the defining data in Definition~\ref{def: fsets}\footnote{In~\cite{PW22}
 an $\f$-inequality $\neq_X$ was defined impredicatively by the existence of an extensional subset 
 $F$ of $\D F(X)$, such that $\neq_X$ is equivalent to $\neq_{(X, F)}$.}.

\begin{definition}\label{def: fsets}
 A completely separated set is a structure $(\B X ; F)$, where $\B X := (X, =_X, \neq_X)$ is in
 $\SetIneq$ and $F$ is an  extensional  subset of $\D F(X)$, such that $\neq_X$ is equal to $\neq_{(X, F)}$ and 
 $\neq_{(X, F)}$ is  tight\footnote{Hence, $=_X$ is also equal to the induced equality $=_{(X, F)}$ on $X$.
 Moreover, $\neq_X$ is tight, but as the tightness of  $\neq_X$ is negativistic, we prefer the positive 
 formulation of the tightness  of $\neq_{(X, F)}$ in the definition of a completely separated set.}.
 We call a function $h \colon (\B X ; F) \to (\B Y ; G)$ affine, or an
 affine arrow, if $g \circ h \in F$, for every $g \in G$
 \begin{center}
\begin{tikzpicture}

\node (E) at (0,0) {$X$};
\node[right=of E] (F) {$Y$};
\node[below=of F] (A) {$\Real$,};

\draw[->] (E)--(F) node [midway,above] {$h$};
\draw[->] (E)--(A) node [midway,left] {$F \ni g \circ h \ $};
\draw[->] (F)--(A) node [midway,right] {$g \in G$};

\end{tikzpicture}
\end{center}
that is $P_F(g \circ h)$, for every $g \in \D F(Y)$ such that $P_G(g)$.
Let $\fV$ be the proper class of completely separated sets, equipped with the equality of $\D V_0^{\neq}$, and  
let $\SetComplSep$ be the full subcategory of $\SetIneq$ of completely separated sets. The category of affine sets
$\SetAffine$ is the subcategory of $\SetComplSep$ with the same objects and only the affine arrows between them.
\end{definition}

By the extensionality of $F$ we get the extensionality of $Q(h) :\TOT \forall_{g \in G}\big(g \circ h \in F\big)$
on $\D F(X, Y)$, hence by separation we define the set of affine arrows from $(\B X ; F)$ to $(\B Y ; G)$, that is
$$\Aff\big((\B X ; F), (\B Y ; G)\big) := \big\{h \in \D F(X, Y) \mid Q(h)\big\}.$$
Following the previous section,
$\big(\B R ; \{\id_{\Real}\}\big)$ and
$\big(\B Z ; U_0(Z)\big)$ are completely separated.
Actually, if $(\B X ; F)$ is in $\SetComplSep$, every element $f$ of $F$ is an affine arrow from 
$(\B X ; F)$ to $\big(\B R ; \{\id_{\Real}\}\big)$.
Clearly, an affine arrow between completely separated sets is a strongly extensional function, or an arrow
in $\SetComplSep$. By Markov's principle every function from $\Real$ to 
$\Real$ is strongly extensional, but not necessarily affine. Using intuitionisitic logic, and 
supposing that $F$ does not include the constant functions on $X$, a constant function from $X$ to $Y$ is strongly 
extensional, but not affine.

\begin{proposition}\label{prp: fcl1}
Let $(\B X ; F) := (X, =_X, \neq_X; F)$ and $(\B Y ; G) := (Y, =_Y, \neq_Y; G)$ be in $\SetComplSep$, and
let $(A, i_A) \subseteq X$.\\[1mm]
\normalfont (i) 
\itshape The product $(\B X \times \B Y ; F \otimes G) := (X \times Y, =_{X \times Y}, \neq_{X} \mathsmaller{\otimes} \neq_Y; F \otimes G)$ is a 
completely separated set, where
$$(x, y) \neq_{X} \mathsmaller{\otimes} \neq_Y (x{'}, y{'}) :\TOT x \neq_{(X,F)} x{'} \vee y \neq_{(Y,G)} y{'} 
\TOT (x, y) \neq_{\mathsmaller{(X \times Y, F \otimes G)}} (x{'}, y{'}),$$
$$F \otimes G := \{f \circ \pr_X \mid f \in F\} \cup \{g \circ \pr_Y \mid g \in G\}.$$
The projections $\pr_X \colon (\B X \times \B Y ; F \otimes G) \to 
(\B X ; F)$
and 
$\pr_Y \colon (\B X \times \B Y ; F \otimes G) \to 
(\B Y ; G)$ are affine arrows\footnote{Hence, $(\B X \times \B Y ; F \otimes G)$ is the product in $\SetAffine$.}.\\[1mm]
\normalfont (ii) 
\itshape The function set $\big(\D F(X, Y), =_{\D F(X, Y)}, \neq_{\D F(X, Y)}; F \to G \big)$ is a 
completely separated set, 
where
$$h \neq_{\D F(X, Y)} h{'} :\TOT \exists_{x \in X}\big(h(x) \neq_{(Y,G)} h{'}(x)\big)  
\TOT h \neq_{\mathsmaller{(\D F(X,Y), F \to G)}} h{'},$$
$$F \to G := \{\phi_{x,g} \mid x \in X \ \& \ g \in G\},$$
$$\phi_{x,g} \colon \D F(X, Y) \to \Real, \ \ \ \phi_{x,g}(h) := g(h(x)), \ \ \ h \in  \D F(X, Y).$$ 
\normalfont (iii) 
\itshape $\big(A, =_A, \neq_A; F \circ i_A \big)$ is completely separated, where 
$a \neq_{A} a{'} :\TOT i_A(a) \neq_{(X,F)} i_A(a{'}) \TOT a \neq_{\mathsmaller{F \circ i_A}} a{'},$
$$F \circ i_A := \{f \circ i_A \mid f \in F\},$$
and $i_A \colon (A, =_A, \neq_{A}; F \circ i_A) \to (X, =_X, \neq_X; F)$ is an affine arrow. 
\end{proposition}

\begin{proof}
 We show only (i) and (ii), as the proof of (iii) is immediate. For (i), if $(x, y), (x{'}, y{'}) \in X \times Y$, then 
 \begin{align*}
  (x, y) \neq_{\mathsmaller{(X \times Y,F \otimes G)}} (x{'}, y{'}) & :\TOT \exists_{h \in F \otimes G}\big(h((x, y)) 
  \neq_{\Real} h((x{'}, y{'}))\big)\\
  & \TOT \exists_{f \in F}\big((f \circ \pr_X)((x, y))  \neq_{\Real} (f \circ \pr_X)((x{'}, y{'}))\big)
  \vee \\
  & \ \ \ \ \ \exists_{g \in G}\big((g \circ \pr_Y)((x, y))  \neq_{\Real} (g \circ \pr_Y)((x{'}, y{'}))\big)\\
  & \TOT:  x \neq_{(X,F)} x{'} \vee y \neq_{(Y,G)} y{'}.
 \end{align*}
The tightness of 
$\neq_{X} \mathsmaller{\otimes} \neq_Y$ follows from 
$(x,y) =_{\mathsmaller{(X \times Y,F \otimes G)}} (x{'}, y{'}) \TOT x =_{(X,F)} x{'} \ \&  \ y =_{(Y,G)} y{'}$ and the 
tightness of $\neq_F$ and $\neq_G$. The fact that the projections are affine arrows follows immediately.\\
(ii) If $h, h{'} \in \D F(X, Y)$, we have that
\begin{align*}
   h \neq_{\mathsmaller{(\D F(X,Y), F \to G)}} h{'} & :\TOT \exists_{x \in X}\exists_{g \in G}\big(\phi_{x,g}(h) 
  \neq_{\Real} \phi_{x,g}(h{'})\big)\\
  & \TOT \exists_{x \in X}\exists_{g \in G}\big(g(h(x)) \neq_{\Real} g(h{'}(x))\big)\\
  & \TOT \exists_{x \in X}\big(h(x) \neq_{(Y,G)} h{'}(x)\big)\\
  & \TOT:  h \neq_{\D F(X, Y)} h{'}.
 \end{align*}
 As $\neq_{(Y,G)}$ is tight, we get
 \begin{align*}
   \ \ \ \ \ \ \ \ \ \ \ \ \ \ \ \ \ \ \ \ \ \ \ \ \ \ \ \ \ \ \ \ \ \ \ \ \ \ \ \ \ \ \ \ \ \ 
   h =_{\mathsmaller{(\D F(X,Y), F \to G)}} h{'} & :\TOT \forall_{x \in X}\forall_{g \in G}\big(\phi_{x,g}(h) 
  =_{\Real} \phi_{x,g}(h{'})\big)\\
  & \TOT \forall_{x \in X}\forall_{g \in G}\big(g(h(x)) =_{\Real} g(h{'}(x))\big)\\
  & \TOT \forall_{x \in X}\big(h(x) =_{(Y,G)} h{'}(x)\big)\\
  & \To \forall_{x \in X}\big(h(x) =_Y h{'}(x)\big)\\
  & \TOT:  h =_{\D F(X, Y)} h{'}. \ \ \ \ \ \ \ \ \ \ \ \ \ \ \ \ \ \ \ \ \ \ \ \ \ \ \ \ \ \ \ \ \ \ \ \ \ \ \ 
     \qedhere
 \end{align*}
 \end{proof}

Next we define a family of completely separated sets indexed by a completely separated 
set $(\B I; K)$.

\begin{definition}\label{def: famoffsets}
If $(\B I ; K)$ is in $\SetComplSep$, 
a \textit{family of completely separated sets}\index{family of sets} indexed by $(\B I; K)$
is a structure $\B S := (\lambda_0, \lambda_1 ; \phi_0, \phi_1)$, where\index{$\lambda_0$}
$(\lambda_0, \lambda_1)$ is an $\B I$-family of sets with an inequality, 
$\phi_0 \colon I \sto \D V_0$ with $\phi_0(i) := F_i$, an extensional subset of 
$\D F(\lambda_0(i))$, and 
\[ \phi_1 \colon \bigcurlywedge_{(i, j) \in D(I)}\D F\big(F_i, F_j\big), 
\ \ \ \phi_1(i, j) =: \phi_{ij} \colon F_i \to F_j, \ \ \ (i, j) \in D(I), \]
such that the following conditions hold:\\[1mm]
\normalfont (a) 
\itshape $x_i \neq_{\lambda_0(i)} x_i{'} \TOT x_i \neq_{(\lambda_0(i), F_i)} x_i{'}$, for every $x_i, x_i{'} 
\in \lambda_0(i)$
and $i \in I$.\\[1mm]
\normalfont (b) 
\itshape $x_i =_{(\lambda_0(i), F_i)} x_i{'} \To x_i =_{\lambda_0(i)} x_i{'}$, for every $x_i, x_i{'} \in \lambda_0(i)$
and $i \in I$.\\[1mm]
\normalfont (c) 
\itshape If $i =_I j$, then the following triangle commutes
\begin{center}
\begin{tikzpicture}

\node (E) at (0,0) {$\lambda_0(j)$};
\node[right=of E] (F) {$\lambda_0(i)$};
\node[below=of F] (A) {$\Real$.};

\draw[->] (E)--(F) node [midway,above] {$\lambda_{ji}$};
\draw[->] (E)--(A) node [midway,left] {$ F_j \ni\phi_{ij}(f_i)   \ \ $};
\draw[->] (F)--(A) node [midway,right] {$f_i \in F_i$};

\end{tikzpicture}
\end{center}
We write $\big(\B \lambda_0(i) ; F_i\big)$ in $\SetComplSep$, 
for every $i \in I$.
\end{definition}

If $(\B X ; F)$ and $(\B Y ; G)$ are in $\SetComplSep$, it is easy 
to define the constant family 
$(\B X ; F)$ over $(\B I ; K)$ and the $\big(\B 2 ; \D F(\D 2)\big)$-family of $(\B X ; F)$ and $(\B Y ; G)$.

\begin{remark}\label{rem: famcss1}
Let  $\B S := (\lambda_0, \lambda_1 ; \phi_0, \phi_1)$ be a family of completely separated sets over  
$(\B I ; K)$.\\[1mm]
\normalfont (i) 
\itshape $\phi_{ij}$ is strongly extensional, for every $(i, j)  \in D(I)$.\\[1mm]
\normalfont (ii) 
\itshape $\lambda_{ij}$ is an affine arrow, for every $(i, j)  \in D(I)$.\\[1mm]
\normalfont (iii) 
\itshape The pair $(\phi_0, \phi_1)$ induces an $\B I$-family of sets with an inequality.\\[1mm]
\normalfont (iv) 
\itshape The pair $(\phi_0, \phi_1)$ induces an $(\B I ; K)$-family of completely separated sets.
\end{remark}

\begin{proof}
(i) If $f_i, f_i{'} \in F_i$, then we have that
\begin{align*}
 \phi_{ij}(f_i) \neq_{F_j} \phi_{ij}(f_i{'}) & :\TOT \exists_{y \in \lambda_0(j)}\big([\phi_{ij}(f_i)](y)
 \neq_{\Real} [\phi_{ij}(f_i{'})](y)\big)\\
 & \TOT \exists_{y \in \lambda_0(j)}\big(f_i(\lambda_{ji}(y)) \neq_{\Real} f_i{'}(\lambda_{ji}(y))\big)\\
 & \To \exists_{x \in \lambda_0(i)}\big(f_i(x) \neq_{\Real} f_i{'}(x)\big)\\
 & \TOT: f_i \neq_{F_i} f_i{'}.
\end{align*}
(ii) It follows immediately by condition (c) in Definition~\ref{def: famoffsets}.\\
(iii) Let $\overline{\phi}_0(i) := \big(F_i, =_{F_1}, \neq_{F_i}\big)$, for every $i \in I$. We use (i), and by condition 
(c) in Definition~\ref{def: famoffsets} we get $\phi_{ii}(f_i) = f_i \circ \lambda_{ii} = f_i$, for every $f_i \in F_i$.
We show the commutativity of the following triangle  
\begin{center}
\begin{tikzpicture}

\node (E) at (0,0) {$F_j$};
\node[right=of E] (F) {$F_k,$};
\node [above=of E] (D) {$F_i$};

\draw[->] (E)--(F) node [midway,below] {$\phi_{jk}$};
\draw[->] (D)--(E) node [midway,left] {$\phi_{ij}$};
\draw[->] (D)--(F) node [midway,right] {$\ \phi_{ik}$};

\end{tikzpicture}
\end{center}
with the commutativity of the corresponding triangle for $(\lambda_0, \lambda_1)$ as follows:
$(\phi_{jk} \circ \phi_{ij})(f_i) = (f_i \circ \lambda_{ji}) \circ \lambda_{kj}$
$= f_i \circ (\lambda_{ji} \circ \lambda_{kj}) = f_i \circ \lambda_{ki} = \phi_{ik}(f_i)$.\\
(iv) Let the quadruple $\big(\B {\phi}_0, \phi_1 ; \theta_0, \theta_1)$, where 
$\theta_0(i) := \widehat{\lambda_0(i)} := \{\widehat{x} \mid x \in \lambda_0(i)\},$ and  
$\widehat{x} \colon F_i \to \Real$ is given by $\widehat{x}(f_i) := f_i(x)$, for every $x \in \lambda_0(i)$ and 
$i \in I$. and $\theta_{ij} \colon \widehat{\lambda_0(i)} \to \widehat{\lambda_0(j)}$ is 
given by $\theta_{ij}\big(\widehat{x}\big) := 
\widehat{\lambda_{ij}(x)}$, for every $x \in \lambda_0(i)$ and $(i,j) \in D(I)$.
It is now straightforward to show conditions (a)--(c) of Definition~\ref{def: famoffsets}. 
\end{proof}

The Sigma-set and the Pi-set of an
$(\B I ; K)$-family $\B S$ of completely separated sets are defined as in Definition~\ref{def: sigmaset}.
According to the next proposition, 
the Pi-set of such a family is in $\SetComplSep$. 

\begin{proposition}\label{prp: fcl2}
If $\B S := (\lambda_0, \lambda_1 ; \phi_0, \phi_1)$ is an $(\B I ; K)$-family of completely separated sets, then
$$\prod_{I}\B S := \bigg(\prod_{i \in I}\lambda_0 (i), =_{\prod_{i \in I}\lambda_0 (i)}, 
\neq_{\prod_{i \in I}\lambda_0 (i)} \ ; \ \bigotimes_{i \in I} F_i\bigg) \in \SetComplSep,$$
where
$$\bigotimes_{i \in I} F_i := \bigg\{f_i \circ \pr_i^{\boldmath \Lambda} \mid f_i \in F_i, \ i \in I\bigg\}.$$

\end{proposition}

\begin{proof}
Proceeding as in the proof of Proposition~\ref{prp: fcl1}(i), we get 
\[\Theta \neq_{\prod_{i \in I}\lambda_0 (i)} \Theta{'} :\TOT \exists_{i \in I}\big(\Theta_i 
\neq_{\lambda_0(i)} \Theta_i{'}\big) 
\TOT \Theta  \neq_{\big(\prod_{i \in I}\lambda_0 (i), \bigotimes_{i \in I} F_i\big)} \Theta{'}.  \qedhere \]
\end{proof}

Similarly, the operation
$\pr_i^{\boldmath \Lambda} \colon  \prod_{i \in I}\lambda_0(i) \sto \lambda_0(i)$, defined by
$\Theta \mapsto \Theta_i,$ is an affine arrow, for every $i \in I$. If $(\B I ; K)$ is discrete, 
and $\B S$ is an $(\B I ; K)$-family of completely separated sets,
then working as in the proof of Proposition~\ref{prp: sigmaset1} we get 
an apartness relation on $\sum_{i \in I}\lambda_0 (i)$, with respect to which
$\pr_1^{\boldmath \Lambda}$ is strongly extensional. Next we examine whether this inequality is
induced by a set of functions.

\begin{proposition}\label{prp: fcl3}
Let $(\B I ; K)$ be discrete, and $\B S := (\lambda_0, \lambda_1 ; \phi_0, \phi_1)$ 
an $(\B I ; K)$-family of completely separated sets.
Let the following extensional subsets of $\D F\big(\sum_{i \in I}\lambda_0 (i)\big):$
$$\widehat{K} := \big\{\widehat{k} \mid k \in K\big\},$$
$$\widehat{k} \colon \sum_{i \in I}\lambda_0 (i) \to \Real, \ \ \ \ \widehat{k}((i, x)) := k(i), \ \ \ 
\ (i, x) \in \sum_{i \in I}\lambda_0 (i),$$
$$\widehat{H} := \bigg\{\widehat{\Phi} \mid \Phi \in \prod_{i \in I}F_i\bigg\},$$
$$\widehat{\Phi} \colon \sum_{i \in I}\lambda_0 (i) \to \Real, \ \ \ \ \widehat{\Phi}(i, x) := \Phi_i(x), \ \ \ 
\ (i, x) \in \sum_{i \in I}\lambda_0 (i).$$
\normalfont (i) 
\itshape For every $(i, x), (j, y) \in \sum_{i \in I}\lambda_0 (i)$, we have that
$$(i, x) \neq_{\mathsmaller{\widehat{K} \cup \widehat{H}}} (j, y) \To (i, x) 
\neq_{\mathsmaller{\sum_{i \in I}\lambda_0(i)}} (j, y).$$
\normalfont (ii) 
\itshape If for every $j \in I$ and every $h \in F_j$ there is $\Phi \in \prod_{i \in I}F_i$, 
such that $\Phi_j =_{\D F(\lambda_0(j)} h$, then for every $(i, x), (j, y) \in \sum_{i \in I}\lambda_0 (i)$, we have that
$$(i, x) \neq_{\mathsmaller{\sum_{i \in I}\lambda_0(i)}} (j, y) \To 
(i, x) \neq_{\mathsmaller{\widehat{K} \cup \widehat{H}}} (j, y).$$
\end{proposition}

\begin{proof}
(i) If there is $k \in K$ such that $\widehat{k}((i, x)) := k(i) \neq_{\Real} k(j) =: \widehat{k}((j, y))$, then 
$i \neq_K j$, and the required inequality follows. Let $\Phi \in \prod_{i \in I}F_i$ such that 
$$\widehat{\Phi}(i, x) := \Phi_i(x) \neq_{\Real} \Phi_j(y) =: \widehat{\Phi}(j, y).$$
If $i \neq_K j$, then we are reduced to the previous case. If $i =_I j$, then 
$$\Phi_i(x) = [\phi_{ji}(\Phi_j)](x) = \Phi_j(\lambda_{ij}(x)) \neq \Phi_j(y),$$
hence $\lambda_{ij}(x) \neq_{\lambda_0(j)} y$, as by the definition of a family of 
completely separated sets the elements of $F_j$, for every $j \in I$, are strongly extensional functions. The required 
inequality follows then immediately.\\
(ii) If $i \neq_K j$, the implication follows trivially. If $i =_I j$ and $\lambda_{ij}(x) \neq_{\lambda_0(j)} y$,
there is $h \in F_j$ such that $h(\lambda_{ij}(x)) \neq_{\Real} h(y)$. Let $\Phi \in \prod_{i \in I}F_i$, 
such that $\Phi_j = h$. As $\Phi_i = \phi_{ji}(\Phi_j) = \Phi_j \circ \lambda_{ij}$, we get
$$\ \ \ \ \ \ \ \ \ \ \ \ \ \ \ \ \widehat{\Phi}(i, x) := \Phi_i(x) =_{\Real} \Phi_j(\lambda_{ij}(x)) =_{\Real}
h(\lambda_{ij}(x)) 
\neq_{\Real} h(y) =_{\Real} \Phi_j(y) =: \widehat{\Phi}(j, y). \ \ \ \ \ \ \ \ \ \ \ \ \ \ \ \ \ \  \qedhere$$
\end{proof}

By Proposition~\ref{prp: fcl3} to get that the Sigma-set of $\B S$ is completely separated, we need 
a discrete, completely separated index-set $(\B I ; K)$, and to suppose that for every $j \in I$,
every element of $F_j$ is ``extended'' to an element of $\prod_{i \in I}F_i$. To overcome these difficulties, we 
introduce the notion of a \textit{global} family of completely separated sets. In contrast to 
Definition~\ref{def: famofsets},
where the transport maps $\lambda_{ij}$ are defined on pairs
$(i, j)$ with $i =_I j$, in a global family the transport maps are defined on every pair
$(i, j)$. 

\begin{definition}\label{def: gfamoffsets}
If $\B I$ is in $\SetIneq$, a global $\B I$-family of sets with an inequality  
is a pair $\boldmath \Lambda^* := (\lambda_0, \lambda_1^*)$, where\index{$\lambda_0$}
$\lambda_0 \colon I \sto \D V_0^{\neq}$, and
\[ \lambda_1^* \colon \bigcurlywedge_{(i, j) \in I \times I}\D F^{\neq}\big(\lambda_0(i), \lambda_0(j)\big), 
\ \ \ \lambda_1^*(i, j) := \lambda_{ij}^*, \ \ \ (i, j) \in I \times I, \]
such that the \textit{global transport maps}\index{transport map of a family of sets} $\lambda_{ij}^*$
\index{transport map of a family of sets}\index{$\lambda_{ij}$}
of $\boldmath \Lambda^*$ satisfy condition \normalfont (a) 
\itshape of Definition~\ref{def: famofsets} and condition \normalfont (b) 
\itshape of Definition~\ref{def: famofsets},
for every $i, j \in I$ with $i =_I j$, and every $k \in I$. If $(\B I ; K)$ is in $\SetComplSep$, an
$(\B I; K)$-family of completely separated sets
is a structure $\B S^* := (\lambda_0, \lambda_1^* ; \phi_0, \phi_1^*)$, where $(\lambda_0, \lambda_1^*)$ is 
a global $\B I$-family of sets with an inequality, 
$\phi_0$ is as in Definition~\ref{def: famoffsets},  
\[ \phi_1^* \colon \bigcurlywedge_{(i, j) \in I \times I}\D F\big(F_i, F_j\big), 
\ \ \ \phi_1^*(i, j) := \phi_{ij}^* \colon F_i \to F_j, \ \ \ (i, j) \in I \times I, \]
and the global maps $\phi_{ij}^*$ satisfy conditions \normalfont (a, b) 
\itshape of Definition~\ref{def: famoffsets} and condition \normalfont (c) 
\itshape of Definition~\ref{def: famoffsets},
for every $i, j \in I$.
The equality and inequality of the Sigma-set of $\Lambda^* (S^*)$ are given by 
\[ (i, x) =_{\mathsmaller{\sum_{i \in I}\lambda_0 (i)}} (j, y) : \TOT i =_I j \ \& \ \lambda_{ij}^* (x) 
=_{\lambda_0 (j)} y, \]
\[ (i, x) \neq_{\mathsmaller{\sum_{i \in I}\lambda_0 (i)}} (j, y) : \TOT i \neq_{(I,K)} j \ \vee \ \lambda_{ij}^*(x) 
\neq_{\lambda_0 (j)} y. \]
The Pi-set of $\Lambda^* (S^*)$ is defined as in Definition~\ref{def: sigmaset}.
\end{definition}

The constant global $\B I$-family of sets with an inequality $\B X$ is defined as the constant $\B I$-family $\B X$,
with $\lambda_{ij}^* := \id_X$, for every $(i, j) \in I \times I$.
If this constant global family is considered, 
then the inequality of its Sigma-set is reduced to the inequality of the product set.
The constant global $(\B I ; K)$-family of completely separated sets $(\B X; F)$ is defined as the constant
$(\B I ; K)$-family $(\B X ; F)$,
with $\phi_{ij}^* := \id_F$, for every $(i, j) \in I \times I$.
To define the global $\B 2$-family $\boldmath \Lambda^{\D 2}(\B X, \B Y)$ of $\B X$ and $\B Y$, we add any two strongly 
extensional functions $\lambda_{01}^* \colon X \to Y$ and  $\lambda_{10}^* \colon Y \to X$, as the following triangles
trivially commute
\begin{center}
\begin{tikzpicture}

\node (E) at (0,0) {$X$};
\node[right=of E] (F) {$Y$};
\node [above=of E] (D) {$X$};
\node[right=of F] (K) {$Y$};
\node [above=of K] (L) {$Y$};
\node [right=of K] (M) {$X$.};

\draw[->] (E)--(F) node [midway,below] {$ \lambda_{01}^*$};
\draw[->] (D)--(E) node [midway,left] {$\id_X$};
\draw[->] (D)--(F) node [midway,right] {$\ \lambda_{01}^*$};
\draw[->] (K)--(M) node [midway,below] {$\lambda_{10}^*$};
\draw[->] (L)--(K) node [midway,left] {$\id_Y$};
\draw[->] (L)--(M) node [midway,right] {$\ \lambda_{10}^*$};

\end{tikzpicture}
\end{center}
To define the global $\big(\B 2 ; \D F(\D 2)\big)$-family of $(\B X ; F)$ and $(\B Y ; G)$, we also add maps 
 $\phi_{01}^* \colon F \to G$ and  $\phi_{10}^* \colon G \to F$, such that the following triangles commute
\begin{center}
\begin{tikzpicture}

\node (E) at (0,0) {$X$};
\node[right=of E] (F) {$Y$};
\node [below=of F] (D) {$\Real$};
\node[right=of F] (S) {};
\node[right=of S] (K) {$Y$};
\node [right=of K] (L) {$X$};
\node [below=of L] (M) {$\Real$,};

\draw[->] (E)--(F) node [midway,above] {$ \lambda_{01}^*$};
\draw[->] (E)--(D) node [midway,left] {$F \ni \phi_{10}^*(g) \ $};
\draw[->] (F)--(D) node [midway,right] {$g \in G$};
\draw[->] (K)--(M) node [midway,left] {$G \ni \phi_{01}^*(f) \ $};
\draw[->] (K)--(L) node [midway,above] {$\lambda_{10}^*$};
\draw[->] (L)--(M) node [midway,right] {$ f \in F$};

\end{tikzpicture}
\end{center} 
and moreover, 
$g \circ \lambda_{01}^* \in F$ and $f \circ \lambda_{10}^* \in G$, 
for every $g \in G$ and $f \in F$, respectively.
Next we show that the Sigma-set of 
a global $(\B I ; K)$-family of completely separated sets is also
completely separated.

\begin{theorem}\label{thm: sigmafset}
Let $(\B I ; K)$ be completely separated and let $\B S^* := (\lambda_0, \lambda_1^*; \phi_0, \phi_1^*)$ be
a global $(\B I ; K)$-family of completely separated sets. Let the extensional subsets $\widehat{K}$ and $\widehat{H}$ of 
$\D F\big(\sum_{i \in I}\lambda_0 (i)\big)$ that were defined in Proposition~\ref{prp: fcl3}.
If the Sigma-set of $\boldmath \Lambda^*$ is equipped with the equality and inequality given 
in Definition~\ref{def: gfamoffsets}, then for every $(i, x), (j, y) \in \sum_{i \in I}\lambda_0 (i)$ we have that
$$(i, x) \neq_{\mathsmaller{\widehat{K} \cup \widehat{H}}} (j, y) \TOT (i, x) 
\neq_{\mathsmaller{\sum_{i \in I}\lambda_0(i)}} (j, y),$$
and hence
$$\sum_{I}\B S^* := \bigg(\sum_{i \in I}\lambda_0 (i), =_{\sum_{i \in I}\lambda_0 (i)}, 
\neq_{\sum_{i \in I}\lambda_0 (i)} \ ; \ \widehat{K} \cup \widehat{H}\bigg) \in \SetComplSep.$$
\end{theorem}

\begin{proof}
To show $(i, x) \neq_{\mathsmaller{\widehat{K} \cup \widehat{H}}} (j, y) \To (i, x) 
\neq_{\mathsmaller{\sum_{i \in I}\lambda_0(i)}} (j, y)$, we work exactly as in the proof of Proposition~\ref{prp: 
fcl3}(i). For the converse implication, we work as in the proof of Proposition~\ref{prp: fcl3}(ii). It suffices 
to show that, if $h \in F_j$ with $h(\lambda_{ij}^*(x)) \neq_{\Real} h(y)$, then there is $\Phi^h \in
\prod_{i \in I}F_i$ with $\Phi^h_j = h$. If $i \in I$, we define $\Phi_i^h := h \circ \lambda_{ij}^*$
\begin{center}
\begin{tikzpicture}

\node (E) at (0,0) {$\lambda_0(i)$};
\node[right=of E] (F) {$\lambda_0(j)$};
\node[below=of F] (A) {$\Real$.};

\draw[->] (E)--(F) node [midway,above] {$\lambda_{ij}^*$};
\draw[->] (E)--(A) node [midway,left] {$ F_i \ni \Phi_i^h   \ \ $};
\draw[->] (F)--(A) node [midway,right] {$h \in F_j$};

\end{tikzpicture}
\end{center}
By the definition of a global $(\B I ; K)$-family of completely separated sets 
we get $\Phi_i^h \in F_i$, for every $i \in I$. Moreover,
$\Phi_j^h = h \circ \lambda_{jj}^* = h \circ \id_{\lambda_0(j)} = h$. To show that $\Phi^h \in
\prod_{i \in I}F_i$, let $i, k \in I$ such that $i =_{(I,K)} k$. By the new  condition (b) satisfied by $\boldmath \Lambda^*$
we have that
$$\phi^*_{ik}\big(\Phi^h_i\big) = \phi^*_{ik}( h \circ \lambda_{ij}^*) =  h \circ \lambda_{ij}^* \circ \lambda_{ki}^* 
=  h \circ \lambda_{kj}^* = \Phi_k^h.$$
Let $(i, x) =_{\mathsmaller{\widehat{K} \cup \widehat{H}}} (j, y)$. By the tightness of $\neq_{(I,K)}$ we have that
$$\forall_{k \in K}\big(\widehat{k}((i,x)) =_{\Real} \widehat{k}((z,y))\big) :\TOT \forall_{k \in K}\big(k(i) =_{\Real}
k(j)\big)
\To i =_I j.$$
If $\Phi \in \prod_{i \in I}F_i$, then
$$\widehat{\Phi}((i,x)) =_{\Real} \widehat{\Phi}((j,y)) :\TOT \Phi_i(x) =_{\Real} \Phi_j(y) \TOT 
\Phi_j(y) =_{\Real} [\phi^*_{ij}(\Phi_j)](x) = \Phi_j(\lambda^*_{ij}(x)).$$
Hence, for every $h \in F_j$ we have that
$h(y) =: \Phi_j^h(y) =_{\Real} \Phi_j^h(\lambda^*_{ij}(x)) := h\big(\lambda^*_{ij}(x))\big).$
By the tightness of $\neq_{F_j}$ we get the required equality $y =_{\lambda_0(j)} \lambda^*_{ij}(x)$.
\end{proof}

The importance of the previous theorem lies on the fact that there is no obvious way to show that the 
canonical inequality of the Sigma-set of a 
global family $\B S^*$ of completely separated sets is an apartness relation using only its definition.
The proof of Proposition~\ref{prp: sigmaset1} cannot be carried out, as the transport maps are between
any two sets of the given family. By showing though, that this inequality is equivalent to the inequality induced by 
$\widehat{K} \cup \widehat{H}$, then by Remark~\ref{rem: f1}(v)
this inequality 
is also an apartness relation on the Sigma-set of $\B S^*$! The reason behind this, seemingly unexpected,
result is the extra information provided by the data associated to the notion of a global family of completely separated
sets. Through the notion of a global family $\B S^*$ of completely separated sets we can also define when a 
dependent function over $\boldmath \Lambda^* (\B S^*)$ is strongly extensional in a way that generalises the 
strong extensionality of (non-dependent) functions.

\begin{definition}\label{def: sedf}
If $\boldmath \Lambda^* (\B S^*)$ is a global $\B I$-family of sets with an inequality, we call
a dependent function $\Phi$ in the Pi-set of 
$\boldmath \Lambda^* (\B S^*)$ strongly extensional, if
$$\forall_{i, j \in I}\big(\lambda_{ij}^*(\Phi_i) \neq_{\lambda_0(j)} \Phi_j \To i \neq_{(I,K)} j\big).$$
\end{definition}

Notice that if $\lambda_{ij}^*(\Phi_i) \neq_{\lambda_0(j)} \Phi_j$, then $\neg(i =_I j)$, as if $i =_I j$, then 
$\lambda_{ij}^*(\Phi_i) =_{\lambda_0(j)} \Phi_j$. The above definition of strong extentionality implies the stronger 
inequality $i \neq_I j$, and clearly generalises the strong extensionality of non-dependent functions, if a 
constant global family of completely separated sets is considered. In analogy to the strong extensionality of 
the first projection, 
we can show that the second projection $\pr_2^{\boldmath \Lambda^*}$
of the Sigma-set of $\boldmath \Lambda^* (\B S^*)$, defined as in Definition~\ref{def: sigmaset},
is a strongly extensional dependent function.

\begin{corollary}\label{cor: sepr2}
 Let $\boldmath \Lambda^*$ be a global $\B I$-family of sets with an inequality.\\[1mm]
\normalfont (i) 
\itshape If $\sigma_0 \colon \sum_{i \in I}\lambda_0 (i) \sto 
\D V_0^{\neq}$ is given by
$\sigma_0(i, x) := \lambda_0(i)$, for every $(i, x) \in \sum_{i \in I}\lambda_0 (i)$, and
$\sigma^*_{(i,x),(j,y)} := \lambda_{ij}^*$, for every 
$(i,x),(j,y) \in \sum_{i \in I}\lambda_0 (i)$, then 
$\boldmath \Sigma^* := (\B \sigma_0, \sigma_1^*)$ is a global $\B I$-family of sets with an inequality. \\[1mm]
\normalfont (ii) 
\itshape The second projection $\pr_2^{\boldmath \Lambda^*}$ is a strongly extensional dependent function over $\boldmath \Sigma^*$.
\end{corollary}

\begin{proof}
 The proof of (i) is immediate. For the proof of (ii), that is
 $$\pr_2^{\boldmath \Lambda^*} \in \prod_{(i,x) \in \sum_{i \in I}\lambda_0 (i)}\sigma_0(i, x) := 
 \prod_{(i,x) \in \sum_{i \in I}\lambda_0 (i)}\lambda_0(i),$$
let $(i,x),(j,y) \in \sum_{i \in I}\lambda_0 (i)$
 such that $(i, x) =_{\mathsmaller{\sum_{i \in I}\lambda_0 (i)}} (j, y)$, that is $i =_{(I,K)} j$ and 
 $\lambda_{ij}^* (x) =_{\lambda_0 (j)} y$. Then,
 \[ \pr_2^{\boldmath \Lambda^*}(j,y) := y
=_{\lambda_0 (j)} \lambda_{ij}^* (x)
:= \sigma_{(i,x),(j,y)}^*\big(\pr_2^{\boldmath \Lambda^*}(i,x)\big).
\]
If $(i,x),(j,y)$ are arbitrary elements of $\sum_{i \in I}\lambda_0 (i)$, such that
$$ \sigma^*_{(i,x),(j,y)}\big(\pr_2^{\boldmath \Lambda^*}((i,x)) \neq_{\sigma_0((j,y))} \pr_2^{\boldmath \Lambda^*}((j,y))
:\TOT \lambda_{ij}^*(x) \neq_{\lambda_0(j)} y,$$
then the inequality  $(i, x) \neq_{\mathsmaller{\sum_{i \in I}\lambda_0 (i)}} (j, y)$ follows immediately.
\end{proof}

\section{The free completely separated set}
\label{sec: free}

Next we define the free completely separated set on a given set $(X, =_X)$.
If $(X, =_X)$ is in $\Set$, then by Remark~\ref{rem: f1}(iii) $(X, =_{(X, F)}, \neq_{(X, F)} ; F)$ is in 
$\SetComplSep$. As $\D F(X)$ is an extensional subset of itself
$$\D F(X) = \big\{f \in \D F(X) \mid f =_{\D F(X)} f\big\},$$
the structure 
$\big(X, =_{(X, \D F(X))}, \neq_{(X, \D F(X))} ; \D F(X)\big)$ is in $\SetComplSep$.

\begin{definition}\label{def: free}
 If $(X, =_X)$ is in $\Set$, the free completely separated set on $(X, =_X)$ is the 
 structure 
$\varepsilon \B X := \big(X, =_{(X, \D F(X))}, \neq_{(X, \D F(X))} ; \D F(X)\big)$.
\end{definition}

\begin{theorem}\label{thm: free1}
\normalfont (i) 
\itshape $\varepsilon \B X$ has the universal property of the free completely separated set on $(X, =_X)$, namely 
there is a function $i_X \colon X \to |\varepsilon \B X|$ such that for every completely separated set $\B Y$ and function 
$h \colon X \to |\B Y|$ there is a unique strongly extensional function $\varepsilon h \colon  |\varepsilon \B X|
\to |\B Y|$ such the following triangle 
commutes
\begin{center}
\begin{tikzpicture}

\node (E) at (0,0) {$X$};
\node[right=of E] (F) {$|\varepsilon \B X|$};
\node[below=of F] (A) {$|\B Y|$.};

\draw[->] (E)--(F) node [midway,above] {$i_X$};
\draw[->] (E)--(A) node [midway,left] {$ h \ \ $};
\draw[->] (F)--(A) node [midway,right] {$\varepsilon h$};

\end{tikzpicture}
\end{center} 
\normalfont (ii) 
\itshape Every completely separated set is the quotient of the free completely separated set over it.\\[1mm]
\normalfont (iii) 
\itshape Let the functor $\Free \colon \Set \to \SetComplSep$, defined by 
$$\Free(X, =_X) := \varepsilon \B X, \ \ \ \Free(f \colon X \to Y) := \varepsilon f \colon |\varepsilon \B X| 
\to |\varepsilon \B Y|$$
\begin{center}
\begin{tikzpicture}

\node (E) at (0,0) {$X$};
\node[right=of E] (F) {$|\varepsilon \B X|$};
\node[below=of E] (A) {$Y$};
\node [right=of A] (D) {$|\varepsilon \B Y|$,};

\draw[->] (E)--(F) node [midway,above] {$i_X$};
\draw[->] (E)--(A) node [midway,left] {$ f $};
\draw[->] (A)--(D) node [midway,below] {$i_Y $};
\draw[->] (F)--(D) node [midway,right] {$\varepsilon f := \varepsilon (i_Y \circ f)$};
\draw[->] (E)--(D) node [left] {};

\end{tikzpicture}
\end{center} 
and let the forgetful functor $\Frg \colon \SetComplSep \to \Set$, defined by
$$\Frg\big(X, =_X, \neq_X ; F\big) := (X, =_X), \ \ \ \Frg(h \colon \B X \to \B Y) := h.$$
Then, $\Free$ is left adjoint to $\Frg$.
\end{theorem}

\begin{proof}
(i) Let $i_X := \id_X$ and $\varepsilon h := h$. That $\id_X$ is a function follows
from Remark~\ref{rem: f1}(i). That
$h \colon  |\varepsilon \B X| \sto |\B Y|$ is a function, follows from the 
fact that every operation from $(X, =_{(X, \D F(X))})$ to $\B Y := (Y, =_Y, \neq_Y ; G)$ is a function; if 
$x, x{'} \in X$ such that $x =_{(X, \D F(X))} x{'}$, then $h(x) =_Y h(x{'}) \TOT h(x) =_{(Y, G)} h(x{'})$, and if 
$g \in G$, then $g \circ h \in \D F(X)$ and the required equality follows from the hypothesis $x =_{(X, \D F(X))} x{'}$.
That $h \colon  |\varepsilon \B X| \sto |\B Y|$ is strongly extensional, is shown as follows: if
$x, x{'} \in X$, and if $g \in G$ such that $g \colon h(x) =_{(Y, G)} h(x{'})$, then $F \ni g \circ h \colon x 
\neq _{(X, \D F(X))} x{'}$. The commutativity of the triangle and the uniqueness of $\varepsilon h$ follow immediately. \\
(ii) If $\B X := \big(X, =_X, \neq_X ; F\big)$, then the identity $\id_X$ from $\varepsilon \B X$ to $\B X$ is the
required surjection.\\
(iii) Clearly, $\Free$ and $\Frg$ are functors. Let $i_{X, \B Y} \colon \Hom(\varepsilon \B X, \B Y)
\to \Hom(X, \Frg(\B Y))$ defined by $i_{X, \B Y} (h) := h$, for every $h \in \Hom(\varepsilon \B X, \B Y)$.
To show that $i_{X, \B Y}$
is well-defined, let $x, x{'} \in X$ with $x =_X x{'}$, hence $x =_{{(X, \D F(X))}} x{'}$. By our hypothesis on $h$,
we get $h(x) =_Y h(x{'})$.  The fact that $i_{X, \B Y}$ is a function, is trivial to show.  Let $j_{X, \B Y} \colon 
\Hom(X, \Frg(\B Y)) \to \Hom(\varepsilon \B X, \B Y)$ defined by $j_{X, \B Y} (h) := \varepsilon h$, for every 
$h \in \Hom(X, \Frg(\B Y))$, where $\varepsilon h$ is determined by the universal property of $\varepsilon \B X$ as follows:
\begin{center}
\begin{tikzpicture}

\node (E) at (0,0) {$X$};
\node[right=of E] (F) {$|\varepsilon \B X|$};
\node[below=of E] (A) {$Y$};
\node [right=of A] (D) {$ \ |\B Y|$.};

\draw[->] (E)--(F) node [midway,above] {$i_X$};
\draw[->] (E)--(A) node [midway,left] {$ h $};
\draw[->] (A)--(D) node [midway,below] {$\id_Y $};
\draw[->] (F)--(D) node [midway,right] {$\varepsilon h := \varepsilon (\id_Y \circ h)$};
\draw[->] (E)--(D) node [left] {};

\end{tikzpicture}
\end{center} 
Clearly, $j_{X, \B Y}$ is a well-defined function, and $(i_{X, \B Y}, j_{X, \B Y})$ witness 
the equality of the two Hom-sets. If $\phi \colon X{'} \to X$, the commutativity of the rectangle
\begin{center}
\begin{tikzpicture}

\node (E) at (0,0) {$ \Hom(\varepsilon \B X, \B Y) $};
\node[right=of E] (K) {};
\node[right=of K] (F) {$ \ \ \Hom(X, \Frg(\B Y))$};
\node[below=of E] (A) {$\Hom(\varepsilon \B X{'}, \B Y)$};
\node [right=of A] (L) {};
\node [right=of L] (D) {$\Hom(X{'}, \Frg(\B Y))$,};

\draw[->] (E)--(F) node [midway,above] {$i_{X, \B Y}$};
\draw[->] (E)--(A) node [midway,left] {$ \Free(\phi)^* $};
\draw[->] (A)--(D) node [midway,below] {$i_{X{'}, \B Y} $};
\draw[->] (F)--(D) node [midway,right] {$\phi^*$};

\end{tikzpicture}
\end{center} 
where $\phi^*(h) := h \circ \phi =: \Free(\phi)^*(h)$, follows easily. Similarly, if $\theta \colon \B Y \to 
\B {Y{'}}$, 
the rectangle 
\begin{center}
\begin{tikzpicture}

\node (E) at (0,0) {$ \Hom(\varepsilon \B X, \B Y) $};
\node[right=of E] (K) {};
\node[right=of K] (F) {$ \ \ \Hom(X, \Frg(\B {Y}))$};
\node[below=of E] (A) {$\Hom(\varepsilon \B X, \B Y{'})$};
\node [right=of A] (L) {};
\node [right=of L] (D) {$\Hom(X, \Frg(\B Y{'}))$};

\draw[->] (E)--(F) node [midway,above] {$i_{X, \B Y}$};
\draw[->] (E)--(A) node [midway,left] {$\theta_* $};
\draw[->] (A)--(D) node [midway,below] {$i_{X, \B Y{'}} $};
\draw[->] (F)--(D) node [midway,right] {$\Frg(\theta)_*$};

\end{tikzpicture}
\end{center} 
is commutative, where $\theta_*(h) := \theta \circ h =: \Frg(\theta)_*(h)$.
\end{proof}

Theorem~\ref{thm: free1}(i) is the $\BST$-analogue to the (intentional) type-theoretic fact that the setoid $(X, =_X)$, where
$=_X \colon X \to X \to \C U$ is the equality type-family  on the type $X \colon \C U$, is the free 
setoid on $X$ (see~\cite{CDPS05}, p.~74). Its simple proof rests on the elimination axiom for $=_X$,
which is a proof-relevant formulation of the least-reflexive property of $=_X$ within intensional $\MLTT$,
a property that cannot be translated directly in $\BST$ (see also~\cite{Pe20}, p.~2 and p.~7). 
As we have noticed after Remark~\ref{rem: f2}, the induced equality $=_{(X, \D F(X))}$ is in $\BST$ the ``smallest'' equality 
on $X$ induced by real-valued functions on $(X, =_X)$.

\section{A set-theoretic Stone-\v{C}ech theorem}
\label{sec: SC}

A completely regular topological space $(X, \mathcal{T})$ is one in which any pair $(x, B)$, where
$B$ is closed and $x \notin B$, is separated by some $f \in C(X, [0, 1])$.  The ring of real-valued,
continuous functions $C(X)$ of a completely regular and $T_{1}$-space $X$,
also known as a \textit{Tychonoff space}, separates the points of $X$, that is
$$\forall_{x, x{'} \in X}\big(\forall_{f \in C(X)}(f(x) = f(x{'})) \Rightarrow x = x{'}\big).$$
The ``sufficiency'' of the completely regular topological spaces in the theory of $C(X)$
is provided by the Stone-\v{C}ech theorem, according to which, for every topological 
space $X$ there exists a completely regular space $\rho X$ and a continuous mapping $\tau_X \colon X \rightarrow \rho X$ 
such that the induced function $f \mapsto \tau^{*}_X(f)$, where
$\tau^{*}_X(f) = f \circ \tau_X$, is a ring isomorphism between $C(\rho X)$ and $C(X)$ (see~\cite{GJ60}, p.41).
\begin{center}
\begin{tikzpicture}

\node (E) at (0,0) {$X$};
\node[right=of E] (F) {$\rho X$};
\node [below=of F] (D) {$\Real$};
\node[right=of F] (P) {};
\node[right=of P] (S) {};
\node[right=of S] (K) {$X$};
\node [right=of K] (L) {$\rho X$};
\node [below=of L] (M) {$Y$};

\draw[->] (E)--(F) node [midway,above] {$ \tau_X$};
\draw[->] (E)--(D) node [midway,left] {$C(X) \ni \tau^{*}_X(f) \ \ $};
\draw[->] (F)--(D) node [midway,right] {$f \in C(\rho X) $};
\draw[->] (K)--(M) node [midway,left] {$C(X, Y) \ni h \ \ $};
\draw[->] (K)--(L) node [midway,above] {$\tau_X$};
\draw[->] (L)--(M) node [midway,right] {$ \rho h \in C(\rho X, Y)$};

\end{tikzpicture}
\end{center} 
Consequently, a functor $\rho \colon \Top \to \crTop$ from the category of topological spaces $\Top$ to 
its subcategory of completely regular topological spaces $\crTop$ is defined, which is a reflector,
that is for every continuous function $h$ from $X$ to a completely regular space $Y$ there is a unique continuous
function $\rho h \colon \rho X \to Y$ such that the above right triangle commutes (see~\cite{He68}, p.~6).
Next we present an abstract version of the Stone-\v{C}ech theorem for completely separated 
sets,
which expresses the corresponding ``sufficiency'' of completely separated sets. 
Topological spaces are replaced by function spaces and completely regular 
topological spaces by completely separated sets\footnote{The adverb ''completely`` in the term completely separated 
set comes from this analogy.}.
The category of function spaces was introduced 
by Ishihara in~\cite{Is13} without restricting though, to extensional subsets of $\D F(X)$.

\begin{definition}\label{def: fs}
A function space is a structure $(X, =_X ; F)$, where $(X, =_X)$ is a set and $F$ is an extensional subset of
$\D F(X)$. If $=_X$ is clear from the context, we also write $(X, F)$. 
In the category $\FunSpace$ of function spaces the arrows are the affine maps, defined in 
Definition~\ref{def: fsets}.
\end{definition}

\begin{theorem}[Stone-\v{C}ech theorem for function spaces and completely separated sets]\label{thm: sc} 
Let $(X, =_X ; F)$ be a function space.\\[1mm]
\normalfont (i)
\itshape  There is a completely separated 
set $\rho_F \B X := \big(X, =_{(X,F)}, \neq_{(X, F)} ; \rho F\big)$ and a function $\tau_X \colon X \to 
|\rho_F \B X|$ such that $(\tau^*_X, \rho_X) \colon F = _{\D V_0} \rho F$, where $\tau^{*}_X \colon \rho F \to 
F$ is defined by $\tau^{*}_X(g) := g \circ \tau_X$, for every $g \in \rho F$, and $\rho_X \colon F \to \rho F$ is
defined by $\rho_X (f) := f$, for every $f \in F$. 
Moreover, every function $f$ in $\rho F$ is strongly extensional.
\begin{center}
\begin{tikzpicture}

\node (E) at (0,0) {$X$};
\node[right=of E] (F) {$|\rho_F \B X|$};
\node [below=of F] (D) {$\Real$};
\node[right=of F] (P) {};
\node[right=of P] (N) {};
\node[right=of N] (S) {};
\node[right=of S] (K) {$X$};
\node [right=of K] (L) {$|\rho_F \B X|$};
\node [below=of L] (M) {$|\B Y|$};

\draw[->] (E)--(F) node [midway,above] {$ \tau_X$};
\draw[->] (E)--(D) node [midway,left] {$F \ni \tau^{*}_X(f) \ \ $};
\draw[->] (F)--(D) node [midway,right] {$f \in \rho F$};
\draw[->] (K)--(M) node [midway,left] {$\Aff\big((X ; F), (Y ; G)\big) \ni h \ \ $};
\draw[->] (K)--(L) node [midway,above] {$\tau_X$};
\draw[->] (L)--(M) node [midway,right] {$ \rho h \in \Aff\big((X ; \rho F), (Y ; G)\big)$};

\end{tikzpicture}
\end{center} 
\normalfont (ii)
\itshape If $\B Y := (Y, =_Y, \neq_Y ; G)$ is in $\SetComplSep$ and $h \colon (X, =_X ; F) \to (Y, =_Y ; G)$ is an
affine map, 
there is a unique affine map 
$\rho h \colon |\rho_F  \B X| \to |\B Y|$, such that
the above right triangle commutes.

\end{theorem}

\begin{proof}
(i) Let the assignment routine $\tau_X \colon X \sto X$ given by the identity rule $\tau_X (x) := x$, 
for every $x \in X$. If $X$ is equipped with the equality 
$=_{(X,F)}$, let $\rho F$ be the following extensional subset of $\D F(X)$
$$\rho F := \big\{f \colon (X, =_F) \to (\Real, =_{\Real}) \mid f \circ \tau_X \in F\big\}.$$
Clearly, the extensionality of $F$ implies the extensionality of $\rho F$, and 
$\rho_F \B X$ is in $\SetComplSep$ as
\begin{align*}
x =_{(X, \rho F)} x{'} & :\TOT \forall_{f \in \rho F}\big(f(x) =_{\Real} f(x{'})\big)\\
&  :\TOT \forall_{f \in \rho F}\big((f \circ \tau_X)(x) =_{\Real} (f \circ \tau_X)(x{'})\big)\\
& \TOT \forall_{f \in F}\big(f(x) =_{\Real} f(x{'})\big)\\
& \TOT: x =_{(X,F)} x{'}.
\end{align*}
In the non-definitional equivalence above we use the fact that if $f \in \rho F$, then $f \circ \tau_X = f \in F$, and
if $f \in F$, then $f \circ \tau_X = f \in F$, and hence $f \in \rho F$.
Similarly we show that 
$x \neq _{(X, \rho F)} x{'} \TOT x \neq_{(X,F)} x{'}$.
The facts that $\tau_X$ is a function and  the pair of functions $(\tau^*_X, \rho_X)$ witness the 
equality $F = _{\D V_0} \rho F$ follow immediately.
The strong extensionality of each $f \in \rho F$ follows immediately from the equivalence of $\neq _{(X, \rho F)}$
and $\neq_{(X,F)}$.\\
(ii) Let $\rho h := h$. To show that $\rho h$ is a function, let $x =_{(X, F)} x{'} \TOT 
x =_{(X, \rho F)} x{'}$. If $g \in G$, then $g(\rho h(x)) :=
g(h(x)) = _{\Real} g(h(x{'})) =: g(\rho h(x{'}))$, as $g \circ h \in F$, hence $g \circ h \in \rho F$.
Consequently, $h(x) =_Y h(x{'})$. The fact that $ \rho h \in \Aff\big((X ; \rho F), (Y ; G)\big)$ follows immediately.
The uniqueness of $\rho h$ follows trivially from the definition of $\tau_X$ and the required commutativity of the above 
right triangle.
\end{proof}

From a given set $(X, =_X)$ and an extensional subset $F$ of $\D F(X)$ we constructed a completely separated set 
$\rho_F \B X$ with the same carrier set $X$ and a larger equality $=_{(X, \rho F)} \TOT =_{(X, F)}$, such that 
all functions that separate the points of $(X, =_{(X, \rho F)})$ are strongly extensional. Of course, 
the inequality $\neq_{(X, F)}$ is tight with respect to the equality  $=_{(X, F)}$.
In the proof of the Stone-\v{C}ech theorem for topological spaces and completely regular spaces $\rho X$ is the the set
of equivalence classes of $X$ with respect to the equivalence relation $x \sim x{'} :\TOT \forall_{f \in C(X)}\big(f(x) =
f(x{'})\big)$. Following~\cite{MRR88}, p.~38, the quotient of a set over an equivalence relation is treated here
as the same totality with the equivalence relation as its new equality. 
Clearly, if $\B X$ is in $\SetComplSep$, the above construction on $(X, =_F)$ induces $\B X$ again.

\begin{proposition}\label{prp: rho}
Let the functor $\rho \colon \FunSpace \to \SetAffine$, defined by
$\rho(X, =_X ; F) := \rho_F \B X$ and $\rho\big(h \colon (X, =_X ; F) \to (Y, =_Y ; G)\big) := \rho h \colon \rho_F \B X 
\to \rho_F \B Y$, where, according to Theorem~\ref{thm: sc}$($ii$)$, $\rho h$ is the unique function that makes 
the following rectangle commutative
\begin{center}
\begin{tikzpicture}

\node (E) at (0,0) {$X$};
\node[right=of E] (K) {};
\node[right=of K] (F) {$|\rho_F \B X|$};
\node[below=of E] (B) {$Y$};
\node[right=of B] (C) {};
\node[right=of C] (A) {$|\rho_F \B Y|$.};

\draw[->] (E)--(F) node [midway,above] {$\tau_X$};
\draw[->] (E)--(B) node [midway,left] {$ h $};
\draw[->] (F)--(A) node [midway,right] {$\rho h$};
\draw[->] (B)--(A) node [midway,below] {$\tau_Y$};
\draw[->] (E)--(A) node [midway,above] {};

\end{tikzpicture}
\end{center}  
If $\Emb \colon \SetAffine \to \FunSpace $ is the
corresponding 
embedding functor, defined by 
$\Emb(\B X) := (X, =_X ; F)$ and $\Emb\big(h \colon (\B X ; F) \to (\B Y ; G)\big) := h$, then $\rho$ is left adjoint to
$\Emb$, and $\Emb$ is left adjoint to $\rho$.
\end{proposition}

\begin{proof}
The fact that $\rho$ is a functor follows immediately by Theorem~\ref{thm: sc}$($ii$)$. From that also follows that 
$\SetAffine$ is reflective in $\FunSpace$, or equivalently, that $\rho$ is left adjoint to
$\Emb$. Next we show that $\Emb$ is left adjoint to $\rho$. We have that 
$$\Hom\big((\Emb(\B X), (Y, =_Y ; G)\big) := \Aff\big((X, =_X ; F), (Y, =_Y ; G)\big),$$
$$\Hom\big(\B X, \rho(Y, =_Y ; G)\big) := \Aff\big((X, =_X, \neq_X ; F), (Y, =_G, \neq_G ; \rho G)\big).$$
If $h \colon X \to Y$ is in $\Aff\big((X, =_X ; F), (Y, =_Y ; G)\big)$, then, for every $g \in \rho G$ we have that
$g \circ h = (g \circ \tau_Y) \circ h \in F$, as $g \circ \tau_Y \in G$ by the definition of $\rho G$. Hence,
$h \in \Aff\big((X, =_X, \neq_X ; F), (Y, =_G, \neq_G ; \rho G)\big)$.
Conversely, if $h \in \Aff\big((X, =_X, \neq_X ; F), (Y, =_G, \neq_G ; \rho G)\big)$ and $g \in G$, we have that
$g \circ h = \rho_Y(g) \circ h \in F$, as $\rho_Y(g) \in \rho G$. Hence $h \in \Aff\big((X, =_X ; F), (Y, =_Y ; G)\big)$.
If $\phi \colon \B X{'} \to \B X$ and $\theta \colon (Y, =_Y ; G) \to (Y{'}, =_{Y{'}} ; G{'})$ are affine 
\begin{center}
\begin{tikzpicture}

\node (E) at (0,0) {$ \Hom(\Emb(\B X), (Y, =_Y ; G)) $};
\node[right=of E] (K) {};
\node[right=of K] (F) {$ \ \Hom(\B X, \rho_G \B Y)$};
\node[below=of E] (A) {$ \ \Hom(\Emb(\B X{'}), (Y, =_Y ; G)$};
\node [right=of A] (L) {};
\node [right=of L] (D) {$ \Hom(\B X{'}, \rho_G \B Y) $};

\draw[->] (E)--(F) node [midway,above] {$\id$};
\draw[->] (E)--(A) node [midway,left] {$ \Emb(\phi)^* $};
\draw[->] (A)--(D) node [midway,below] {$\id$};
\draw[->] (F)--(D) node [midway,right] {$\phi^*$};

\end{tikzpicture}
\end{center} 
\begin{center}
\begin{tikzpicture}

\node (E) at (0,0) {$  \Hom(\Emb(\B X), (Y, =_Y ; G)) \ \ $};
\node[right=of E] (K) {};
\node[right=of K] (F) {$ \ \ \ \Hom(\B X, \rho_G \B Y)$};
\node[below=of E] (A) {$ \ \ \Hom(\Emb(\B X), (Y{'}, =_{Y{'}} ; G{'}))$};
\node [right=of A] (L) {};
\node [right=of L] (D) {$ \Hom(\B X, \rho_{G{'}} \B Y{'})$};

\draw[->] (E)--(F) node [midway,above] {$\id$};
\draw[->] (E)--(A) node [midway,left] {$\theta_* $};
\draw[->] (A)--(D) node [midway,below] {$\id$};
\draw[->] (F)--(D) node [midway,right] {$(\rho \theta)_*$};

\end{tikzpicture}
\end{center} 
the commutativity of the above rectangles is straightforward to show.
\end{proof}

It is not often the case that we have functors like $\rho$ and $\Emb$, such that $\rho \dashv \Emb \dashv \rho$.
For example, although in Theorem~\ref{thm: free1}(iii) we showed that $\Free \dashv \Frg$, we cannot show that 
$\Frg \dashv \Free$. Because of the relations $\rho \dashv \Emb \dashv \rho$, we have that $\rho$ preserves all limits 
and colimits. For example, as $(X \times Y, =_{X \times Y} ; F \otimes G)$ is the product of
$(X, =_X ; F)$ and $(Y, =_Y,; G)$ in $\FunSpace$, we have that $\rho_{F \otimes G}(\B X \times \B Y) =
\rho_F(\B X) \times \rho_G(\B Y)$.

\section{A set-theoretic Tychonoff embedding theorem}
\label{sec: TET}

According to the classical Tychonoff embedding theorem, a $T_1$ topological space is completely regular if and only if it 
is topologically embedded into a product of $[0, 1]$ equipped with its standard topology. Here we give
a purely set-theoretic 
formulation of this result within $\BST$. As the separating functions considered here take values in $\Real$, the 
corresponding product will be a (dependent) product of $\Real$. As in the previous section, we replace 
complete regularity by complete separation.

\begin{lemma}\label{lem: lel}
 Let  $\B X, \B I$ be in $\SetIneq$ and $\boldmath \Lambda := (\lambda_0, \lambda_1)$ a family of 
 sets with an inequality over $\B I$. 
 If $\boldmath M := (\B \mu_0, \mu_1)$, where $\B \mu_0 \colon I \sto \D V_0^{\neq}$ is defined by 
 $\overline{\mu}_0(i) := \big(\D F(X, \lambda_0(i)), =_{\D F(X, \lambda_0(i))}, 
 \neq_{\D F(X, \lambda_0(i))}\big)$
 and 
 \[ \mu_1 \colon \bigcurlywedge_{(i, j) \in D(I)}\D F^{\neq}\big(\D F(X, \lambda_0(i)), \D F(X, \lambda_0(j))\big), 
\ \ \ \mu_1(i, j) := \mu_{ij}, \ \ \ (i, j) \in D(I), \]
$$\mu_{ij} \colon \D F(X, \lambda_0(i)) \to \D F(X, \lambda_0(j)), \ \ \ \ 
\mu_{ij}(\phi) := \lambda_{ij} \circ \phi, \ \ \ \phi \in \D F(X, \lambda_0(i))$$
\begin{center}
\begin{tikzpicture}

\node (E) at (0,0) {$X$};
\node[right=of E] (F) {$\lambda_0(i)$};
\node[below=of F] (A) {$\lambda_0(j)$,};

\draw[->] (E)--(F) node [midway,above] {$\phi$};
\draw[->] (E)--(A) node [midway,left] {$ \mu_{ij}(\phi)  \ \ $};
\draw[->] (F)--(A) node [midway,right] {$\lambda_{ij}$};

\end{tikzpicture}
\end{center}
then $\boldmath M$ is a family of sets with an inequality over $\B I$.
\end{lemma}

\begin{proof}
If $i \in I$, then $\mu_{ii}(\phi) := \lambda_{ii} \circ \phi = \id_{\lambda_0(i)} \circ \phi = \phi$,
for every $\phi \in \D F(X, \lambda_0(i))$. If 
$i =_I j =_I k$, then
\begin{center}
\begin{tikzpicture}

\node (E) at (0,0) {$\D F(X, \lambda_0(j))$};
\node[right=of E] (F) {$\D F(X, \lambda_0(k)).$};
\node [above=of E] (D) {$\D F(X, \lambda_0(i))$};

\draw[->] (E)--(F) node [midway,below] {$\mu_{jk}$};
\draw[->] (D)--(E) node [midway,left] {$\mu_{ij}$};
\draw[->] (D)--(F) node [midway,right] {$ \ \ \mu_{ik}$};

\end{tikzpicture}
\end{center}
$(\mu_{jk} \circ \mu_{ij})(\phi) := \mu_{jk}\big(\lambda_{ij} \circ \phi)
:= \lambda_{jk} \circ \big(\lambda_{ij} \circ \phi)
= \lambda_{ij} \circ \phi
 =: \mu_{ik}(\phi),$
for every $\phi \in \D F(X, \lambda_0(i))$. To show that $\mu_{ij}$ is strongly extensional, let 
$\phi, \phi{'} \in \D F(X, \lambda_0(i))$. As $\lambda_{ij}$ is strongly extensional, we have that
\begin{align*}
\ \ \ \ \ \ \ \ \ \ \ \ \ \ \ \ \ \ \ \ \ \ \ \ \ \ \ \ \ \ \ 
\mu_{ij}(\phi) \neq_{\D F(X, \lambda_0(i))} \mu_{ij}(\phi{'}) & :\TOT  \lambda_{ij} \circ \phi 
\neq_{\D F(X, \lambda_0(i))} \lambda_{ij} \circ \phi{'}\\
& :\TOT \exists_{x \in X}\big(\lambda_{ij}(\phi(x)) \neq_{\lambda_0(j)} \lambda_{ij}(\phi(x{'}))\big)\\
& \To \exists_{x \in X}\big(\phi(x) \neq_{\lambda_0(i)} \phi(x{'})\big)\\
& :\TOT \phi \neq_{\D F(X, \lambda_0(i))} \phi{'}. \ \ \ \ \ \ \ \ \ \ \ \ \ \ \ \ \ \ \ \ \ \ \ \ \ \ \ 
\ \ \ \ \ \ \ \ \ \ \ \ \  \qedhere
\end{align*}
\end{proof}

\begin{theorem}[Embedding lemma for sets with an inequality]\label{thm: el}
 Let $\B X, \B I$ be in $\SetIneq$ and $\boldmath \Lambda := (\lambda_0, \lambda_1)$ a family of sets with an inequality
 over $\B I$.
 Let also $H \in \prod_{i \in I}\mu_0(i)$, where  $\boldmath M := (\mu_0, \mu_1)$ is the family of sets with inequality 
 over $\B I$ defined in Lemma~\ref{lem: lel}. 
 If 
 $$\prod_{I}\Lambda := \bigg(\prod_{i \in I}\lambda_0(i), =_{\prod_{i \in I}\lambda_0(i)}, 
 \neq_{\prod_{i \in I}\lambda_0(i)}\bigg),$$
let the assignment routine 
 $$e^H \colon X \sto \prod_{i \in I}\lambda_0(i), \ \ \ x \mapsto e^H(x),$$
 $$\big[e^H(x)\big]_i := H_i(x), \ \ \ i \in I.$$
 \normalfont (i)
\itshape $e^H$ is a well-defined function.\\[1mm]
 \normalfont (ii)
\itshape Let the extensional subset $\Phi := \{e^H\}$ 
of $\D F\big(X, \prod_{i \in I}\lambda_0(i)\big)$. 
If the induced inequality $\neq_{(X,\Phi)}$ on $X$ is tight, then $e^H$ is an
embedding, and if $H_i$ is strongly extensional, for every $i \in I$, then $e^H$ is strongly extensional. 
%
\end{theorem}
 
\begin{proof}
(i) First we show that $e^H$ is well-defined, that is $e^H \in \prod_{i \in I}\lambda_0(i)$. If $i =_I j$, then
$$\big[e^H(x)\big]_j := H_j(x) =_{\lambda_0(j)} \big[(\mu_{ij}(H_i)\big](x) := [\lambda_{ij} \circ H_i](x)
:= \lambda_{ij}\big(H_i(x)\big) =: \lambda_{ij}\big(\big[e^H(x)\big]_i\big].$$
If $x =_X x{'}$, then $e^H(x) =_{\prod_{i \in I}\lambda_0(i)} e^H(x{'})$, as for every $i \in I$ we 
have that $H_i(x) =_{\lambda_0(i)} H_i(x{'})$.\\
(ii) The tightness of $\neq_{(X,\Phi)}$ means that $x =_{(X,\Phi)} x{'} \To x =_X x{'}$, for every $x, x{'} \in X$, where
\begin{align*}
x =_{(X,\Phi)} x{'}  & :\TOT  e^H(x) =_{\prod_{i \in I}\lambda_0(i)} e^H(x{'})\\
& :\TOT \forall_{i \in I}\big(\big[e^H(x)\big]_i =_{\lambda_0(i)} \big[e^H(x{'})\big]_i\\
& :\TOT \forall_{i \in I}\big(H_i(x) =_{\lambda_0(i)} H_i(x{'})\big).
\end{align*}
Hence, $e^H(x) =_{\prod_{i \in I}\lambda_0(i)} e^H(x{'}) \To x =_X x{'}$ follows immediately.
If each $H_i$ is strongly extensional, 
then
\begin{align*}
 \ \ \ \ \ \ \ \ \ \ \ \ \ \ \ \ \ \ \ \ \ \ \ \ \  e^H(x) \neq_{\prod_{i \in I}\lambda_0(i)} e^H(x{'}) & :\TOT 
 \exists_{i \in I}\big(\big[e^H(x)\big]_i 
 \neq_{\lambda_0(i)} \big[e^H(x{'})\big]_i\\
 & :\TOT \exists_{i \in I}\big(H_i(x) \neq_{\lambda_0(i)} H_i(x{'})\big)\\
 & \To x \neq_X x{'}. \ \ \ \ \ \ \ \ \ \ \ \ \ \ \ \ \ \ \ \ \ \ \ \ \ \ \ \ \ \ \ \
 \ \ \ \ \ \ \ \ \ \ \ \ \ \ \ \ \ \ \ \ \ \ \ \ \ \ \ \ \qedhere
\end{align*}
\end{proof}

\begin{definition}\label{def: dual}
 If $(X, =_X ; F)$ is a function space, its dual completely separated set is the structure 
 $\B X^* := (F, =_F, \neq_F ; \widehat{X})$, where $=_F$ and $\neq_F$ are induced by $\D F(X)$, and 
 $$\widehat{X} := \{\widehat{x} \mid x \in X\}, \ \ \ \ 
 \widehat{x} \colon F \to \Real, \ \ \ \widehat{x}(f) := f(x); \ \ \ f \in F, \  x \in X.$$
 The contravariant functor $\Dual \colon \FunSpace^{\op} \to \SetAffine$ is defined by $\Dual(X, =_X ; F) := \B X^*$ and
 $\Dual\big(h \colon (X, =_X ; F) \to (Y, =_Y ; G)\big) := h^* \colon \B Y^* \to \B X^*$,
 where $h^*(g) := g \circ h \in F$,
 for every $g \in G$.
\end{definition}

Clearly, one can ``identify'' $(X, =_X ; F)$ with $(\widehat{X}, =_{\widehat{X}} ; \widehat{F})$, where $ \widehat{F}
:= \{\widehat{f} \mid f \in F\}$ and $\widehat{f} (\widehat{x}) := \widehat{x}(\widehat{f}) := f(x)$, for 
every $f \in F$ and
$x \in X$, as the maps $x \mapsto \widehat{x}$ and $\widehat{x} \mapsto x$ form an iso in the category $\FunSpace$.

\begin{remark}\label{rem: beforetet}
Let $(X, =_X ; F)$ be a function space.\\[1mm]
\normalfont (i)
\itshape Let $\B \lambda_0(f) := \B R$, for every $f \in F$, and  
$\lambda_{fg} := \id_{\Real}$, for every $f, g \in F$
such that $f =_F g$. Moreover, let $\phi_0(f) := \{\id_{\Real}\}$, for every $f \in F$ and, if $f =_F g$,
let $\phi_{fg} \colon \{\id_{\Real}\} \to \{\id_{\Real}\}$ be given by the identity rule. Then the structure
$\B S := (\lambda_0, \lambda_1 ; \phi_0, \phi_1)$ is a family of complemented sets over the dual completely separated set 
$\B X^* := (\B F ; \widehat{X})$ of the function space $(X, =_X ; F)$.\\[1mm]
\normalfont (ii)
\itshape If $\B R^{F} := \big(\Real^{F}, =_{\Real^{F}}, \neq_{\Real^{F}} ; 
\bigotimes_{f \in F}\{\id_{\Real}\}\big),$
where
$$\bigotimes_{f \in F} \{\id_{\Real}\} := \big\{\id_{\Real} \circ \pr_f^{\boldmath \Lambda} = \pr_f^{\boldmath \Lambda}
\mid f \in F\big\},$$
then $\B R^{F} \in \SetComplSep$.
\end{remark}

\begin{proof}
 The proof of case (i) is straightforward and case (ii) follows from case (i) and Proposition~\ref{prp: fcl2}.
\end{proof}

Clearly, if $(X, =_X)$ is in $\Set$, and 
$\neq_{(X,F)}$ is tight, 
then $(X, =_X, \neq_{(X, F)} ; F)$ is in $\SetComplSep$.

\begin{theorem}[Tychonoff embedding theorem for function spaces and completely separated sets]\label{thm: tet} 
Let $(X, =_X ; F)$ be a function space.\\[1mm]
\normalfont (i)
\itshape If the induced inequality $\neq_{(X,F)}$ on $X$ is tight, then there is an affine embedding $($injection$)$ 
of the 
completely separated set $(X, =_X, \neq_{(X, F)} ; F)$ into the completely separated set $\B R^{F}$.\\[1mm]
 \normalfont (ii)
\itshape 
If $e \colon (X, =_X ; F) \to (\Real^{F}, =_{\Real^{F}} ; \bigotimes_{f \in F}\{\id_{\Real}\})$ is an affine 
embedding, then the  induced inequality $\neq_{(X,F)}$ on $X$ is tight, and hence
$(X, =_X, \neq_{(X, F)} ; F)$ is in $\SetComplSep$.
\end{theorem}
 
\begin{proof} 
(i) Let $\boldmath M := (\mu_0, \mu_1)$ be the family of sets over $\B F$ from Lemma~\ref{lem: lel}
that corresponds to the constant $\B F$-family of sets with an inequality $(\lambda_0, \lambda_1)$ 
in Remark~\ref{rem: beforetet}(i).
Let the dependent function 
$$H(F) \colon \prod_{f \in F}\mu_0(f) :=  \prod_{f \in F}\D F(X, \Real),$$
$$H(F)_f := f,  \ \ \ \ f \in F.$$
If $f =_F g$, then $H^F_g := g =_F f =: H(F)_f := \id_{\Real} \circ H(F)_f := \lambda_{fg} \circ H(F)_f 
:= \mu_{fg}\big(H(F)_f\big)$, hence $H(F)$ is well defined.
By Theorem~\ref{thm: el}(i) we get the function 
$$e^{H(F)} \colon X \to \prod_{f \in F}\Real =: \Real^F, \ \ \ x \mapsto e^{H(F)}(x),$$
$$\big[e^{H(F)}(x)\big]_f := H(F)_f(x) := f(x), \ \ \ f \in F, \ x \in X.$$
Next we show that the inequality $\neq_{(X,\Phi)}$ on $X$, where $\Phi := \{e^{H(F)}\}$, is tight. If $x, x{'} \in X$, then
\begin{align*}
 x =_{(X,\Phi)} x{'} & :\TOT e^{H(F)}(x) =_{\Real^F} e^{H(F)}(x{'})\\
 & :\TOT \forall_{f \in F}\big(\big[e^{H(F)}(x)\big]_f =_{\Real} \big[e^{H(F)}(x{'})\big]_f\big)\\
 & :\TOT \forall_{f \in F}\big(f(x) =_{\Real} f(x{'})\big)\\
 & :\TOT x =_F x{'}\\
 & \To x =_X x{'},
\end{align*}
where the last implication follows from the tightness of $\neq_{(X,F)}$. To show that $e^{H(F)}$ is affine, it
suffices to show that $\pr_f^{\boldmath \Lambda} \circ e^{H(F)} \in F$, for every $f \in F$. 
By the definition of $e^{H(F)}$ though, we have that $\pr_f^{\boldmath \Lambda} \circ e^{H(F)} =_F f \in F$.\\
%
%
(ii) If $x, x{'} \in X$ such that $x =_{(X,F)} x{'}$. we show that $x =_X x{'}$. As $e$ is an embedding, 
it suffices to show that $e(x) =_{\Real^{F}} e(x{'}) \TOT e(x) =_{\bigotimes_{f \in F}\{\id_{\Real}\}} e(x{'})
\TOT \forall_{f \in F}\big(\pr_f^{\boldmath \Lambda}(e(x)) =_{\Real} \pr_f^{\boldmath \Lambda}(e(x{'}))\big)$. 
As $e$ is affine, the composition $ \pr_f^{\boldmath \Lambda} \circ e$ is in $F$, for every $f \in F$, and we use the hypothesis
$x =_{(X,F)} x{'}$.
%
\end{proof}

Notice that in the above proof we avoided negation completely. Theorem~\ref{thm: tet} offers a 
criterion for the generation of a completely separated set $(X, =_X, \neq_{(X, F)} ; F)$ 
from a given function space $(X, =_X ; F)$.

\section{Concluding remarks}
\label{sec: concl}

In this paper we begun the study of completely separated sets within $\BST$, realising both 
the attitude within $\CM$ of using positive definitions, instead of negative ones, and the attitude of 
employing functions, instead of sets, whenever that is possible.
Completely separated sets are equipped with a positive notion of an inequality, as their given inequality 
is equivalent to the inequality induced by a set of real-valued functions.
Showing that a set with an inequality $(X, =_X, \neq_X)$ is completely separated, by finding 
an extensional subset $F$ of $\D F(X)$ such that
$=_X$ is equivalent to $=_{(X,F)}$ and $\neq_X$ is equivalent to $\neq_{(X,F)}$, is a method of proving that $\neq_X$ is
an apartness, and hence an extentional relation. As we saw in the case of Theorem~\ref{thm: sigmafset}, to prove directly
that $\neq_X$ is an apartness relation can be non-trivial.

The various distinctions that are made possible within $\CM$ facilitate the definition of interesting, new categories of sets,
such as the category $\SetIneq$ of sets with an inequality, and the categories $\SetComplSep$ and $\SetAffine$ of 
completely separated and affine sets, respectively. If $(\B C \subseteq \B D) \ \B C \subset \B D$ denotes that $\B C$ is 
a (full) subcategory of $\B D$, we have that
$$\SetAffine \subset \SetComplSep \subseteq \SetIneq.$$
We defined the corresponding notions of families of sets in these categories and their Pi- and
Sigma-sets, extending\footnote{We also added two more open-ended universes in $\BST$, the universe $\D V_0^{\neq}$ of 
predicative sets with an inequality and the universe of completely separated sets $\fV$. Clearly, both of them are
proper classes.} in this way our previous work~\cite{Pe20}.
We also introduced the notions of a global family of sets with an inequality over an index-set with an inequality and
of a global family of
completely separated sets over an index-completely separated set, in order to describe the Sigma-set of the latter.
Global families helped us to formulate a notion of strong extensionality for dependent functions that generalises the
strong extensionality of non-dependent functions.
The free completely separated set on a given set gave us a way to translate into $\BST$ the freeness of the
equality setoid on a given type within intentional $\MLTT$.
We also provided purely set-theoretic versions of the classical Stone-\v{C}ech theorem and the Tychonoff 
embedding theorem for completely regular spaces, by replacing topological spaces with 
function spaces and completely regular spaces with completely separated sets
$$\frac{\SetComplSep}{\FunSpace} \approx \frac{\crTop}{\Top}.$$
Although we cannot show constructively that all real-valued functions on a function space $(X, =_X ; F)$
are strongly extensional, by Theorem~\ref{thm: sc} we can always find a new equality on $X$ 
such that all elements of $F$ become strongly extensional! The positive formulation of our basic concepts also 
helped us to avoid negation in the proofs of Theorems~\ref{thm: sc} and~\ref{thm: tet}.
If the set $F$ in a function space is closed under the minimum operation, namely $P_F(f) \ \& \ P_F(g) 
\To P_F(f \wedge g)$, where $f \wedge g := \min\{f, g\}$, and if for every $x \in X$ there is some $f \in F$ such that $f(x) > 0$, then 
the family $(U(f))_{f \in F}$, where 
$$U(f) = \{x \in X \mid f(x) > 0\},$$
is a base for a topology of open sets on $X$. The tightness of the inequality $\neq_{(X,F)}$ is equivalent then to the 
topological fact that each singleton $\{x\}$ in $X$ is a closed set (see Proposition 3.5 in~\cite{Pe22b}). 
Theorems~\ref{thm: sc} and~\ref{thm: tet} can also be seen as purely set-theoretic versions of the
constructive Stone-\v{C}ech theorem and the Tychonoff 
embedding theorem for Bishop spaces in~\cite{Pe15, Pe15a}.

The categories $\Set, \SetIneq, \SetComplSep$ and $\SetAffine$ need to be studied further,
and to be connected with the elaboration of abstract category theory within 
$\BST$ in~\cite{Pe22f}. 
The theory of (global) families of completely separated sets is expected to be 
developed in line of the general theory of families
of sets in~\cite{Pe20}. The arrows between these families, the corresponding 
distributivity of the Pi-set over the Sigma-set, and the families of subsets of a
completely separated set, are some of the topics to be studied along the lines of~\cite{Pe20}.
Inequalities are easier to study in $\BST$, rather than in (intensional) $\MLTT$.
Also, the notion of a global family of completely separated sets over an index-completely separated set 
seems to have no (intentional)
type-theoretic analogue. The question whether the above categories of sets with an 
inequality, and the notions of their indexed families studied here, can be grasped
by intensional $\MLTT$, seems to us both interesting and hard.

\end{document}